\documentclass[11pt]{amsart}

\usepackage[margin=1in]{geometry}
\usepackage{amsfonts}
\usepackage{enumerate}
\usepackage{graphicx}
\usepackage{float}
\usepackage{array}
\usepackage{hyperref}
\usepackage{amssymb}
\usepackage{url}
\usepackage{mathtools,nccmath}
\usepackage[font=small,skip=0ex]{caption}

\usepackage{tikz}
\usetikzlibrary{matrix,positioning}

\usepackage{hyperref,xcolor}
\hypersetup{
colorlinks,
linkcolor={red!50!black},
citecolor={blue!50!black},
urlcolor={blue!80!black}
}

\DeclareSymbolFont{bbold}{U}{bbold}{m}{n}
\DeclareSymbolFontAlphabet{\mathbbold}{bbold}

\theoremstyle{theorem}
\newtheorem{theorem}{Theorem}[section]
\newtheorem{proposition}[theorem]{Proposition}
\newtheorem{lemma}[theorem]{Lemma}

\theoremstyle{definition}
\newtheorem{definition}[theorem]{Definition}
\newtheorem{example}[theorem]{Example}

\theoremstyle{remark}

\numberwithin{theorem}{section}
\numberwithin{equation}{section}

\DeclareMathOperator{\cHu}{\mathcal{H}}

\newcommand{\Jac}{\nabla}
\DeclareMathOperator{\codim}{codim}

\DeclareMathOperator{\diag}{diag}

\newcommand{\FP}{\textsc{fp}}
\newcommand{\DC}{\textsc{dc}}

\DeclareMathOperator{\EP}{\mathcal{EP}}
\DeclareMathOperator{\EV}{\mathcal{EV}}

\newcommand{\ii}{\mathrm{i}}

\newcommand{\wv}{\boldsymbol{w}}
\newcommand{\xv}{\boldsymbol{x}}  
\newcommand{\yv}{\boldsymbol{y}}  
\newcommand{\zv}{\boldsymbol{z}}
\newcommand{\pv}{\boldsymbol{p}}
\newcommand{\qv}{\boldsymbol{q}}
\newcommand{\nl}{\ell} 
\newcommand{\Lam}{\boldsymbol{\lambda}}
\newcommand{\tp }{{\scriptscriptstyle\mathsf{T}}}

\newcommand{\blue}[1]{#1}

\begin{document}

\title{Fiber product homotopy method for multiparameter eigenvalue problems}

\author[J.~Rodriguez]{Jose Israel Rodriguez}
\address{Department of Mathematics,
University of Wisconsin, Madison, WI}
\email[corresponding author]{jrodriguez43@wisc.edu}

\author[J.-H.~Du]{Jin-Hong Du}
\address{Department of Statistics,    University of Chicago, Chicago, IL}
\email{dujinhong@uchicago.edu}

\author[Y.~You]{Yiling You}
\address{Department of Mathematics, University of California, Berkeley, CA}
\email{yiling.you@berkeley.edu}

\author[L.-H.~Lim]{Lek-Heng Lim}
\address{Computational and Applied Mathematics Initiative,
University of Chicago, Chicago, IL}
\email{lekheng@uchicago.edu}


\maketitle
\begin{abstract}
We develop a new homotopy method for solving multiparameter eigenvalue problems (MEPs) called the fiber product homotopy method.
\blue{For a $k$-parameter eigenvalue problem with matrices of sizes $n_1,\dots ,n_k = O(n)$, fiber product homotopy method requires deformation of $O(1)$ linear equations, while existing homotopy methods for MEPs require $O(n)$ nonlinear equations.}
We show that the fiber product homotopy method theoretically finds all eigenpairs of an MEP with probability one. It is especially well-suited for \blue{dimension-deficient} singular MEPs, a weakness of all other existing methods, as the fiber product homotopy method is provably convergent with probability one for \blue{such} problems as well, a fact borne out by numerical experiments.
More generally, our numerical experiments indicate that the fiber product homotopy method significantly outperforms the standard Delta method in terms of accuracy, with consistent backward errors on the order of $10^{-16}$, even for \blue{dimension-deficient} singular problems, and without any use of extended precision. In terms of speed, it significantly outperforms previous homotopy-based methods on all problems and outperforms the Delta method on larger problems, and is also highly parallelizable. We show that the \emph{fiber product MEP} that we solve in the fiber product homotopy method, although mathematically equivalent to a standard MEP, is typically a much better conditioned problem.
\end{abstract}

\section{Introduction}

A multiparameter eigenvalue problem (MEP) is, in an appropriate sense, a system of linear equations
\begin{equation}\label{eq:lin}
\begin{alignedat}{8}
a_{11} x_1  &+  a_{12} x_2  &&+ \dots &&+  a_{1k} x_k  &&= b_1, \\
a_{21} x_1  &+  a_{22} x_2  &&+ \dots &&+  a_{2k} x_k  &&= b_2,\\
\MTFlushSpaceAbove
&\vdotswithin{+} && \vdotswithin{+}&& \vdotswithin{+}&& \vdotswithin{=}
\MTFlushSpaceBelow
a_{k1} x_1  &+  a_{k2} x_2  &&+ \dots &&+  a_{kk} x_k  &&=b_k,
\end{alignedat}
\end{equation}
where the coefficients $a_{ij}$'s and  $b_i$'s are \emph{matrices}, and where equality is interpreted to mean on a point in a product of projective spaces (this will be made precise later). These coefficients are square matrices but are of different dimensions in general, so one may not usually regard \eqref{eq:lin} as a linear system over a matrix ring. There is a rich mathematical theory behind MEP \cite{Atkinson,Atkinson2011} that places it at the crossroad of linear and multilinear algebra, ordinary and partial differential equations, spectral theory and Sturm--Liouville theory, among other areas.  The problem appeared as early as 1836 in the works of Sturm and Liouville on periodic heat flow in a bar, and was studied over the years by many: Klein, Lam\'e, Heine, Stieltjes, Pell, Carmichael, Bocher, Hilbert among them (see \cite[Preface]{Atkinson} and \cite[Chapter~1]{Atkinson2011}).

An MEP encompasses many known types of eigenvalue problems: Standard eigenvalue problems $A\xv = \lambda \xv$; generalized eigenvalue problems $A\xv = \lambda B\xv$; quadratic eigenvalue problems $(\lambda^2 A+\lambda B+C)\xv = 0$; polynomial eigenvalue problems $(\lambda^mA_m+\lambda^{m-1}A_{m-1}+\cdots+A_0)\xv = 0$; quadratic two-parameter eigenvalue problems
\begin{align*}
(A_{00}+  \lambda A_{10}+\mu A_{01}+\lambda^{2}A_{20}+\lambda\mu A_{11}+\mu^{2}A_{02})\xv_1 &= 0,\\
(B_{00}+  \lambda B_{10}+\mu B_{01}+\lambda^{2}B_{20}+\lambda\mu B_{11}+\mu^{2}B_{02}) \xv_2 &= 0;
\end{align*}
may all be reduced to mathematically equivalent MEPs.

Nevertheless MEP remains in the blind spot of most modern mathematicians, whether pure or applied. This is not for its lack of applications; as we pointed out, the problem in fact originated from a study of heat flow, and we will see yet other applications of MEP in Section~\ref{sec:implement} and that it contains eigenvalue problem and linear system, both ubiquitous in science and engineering, as special cases. We think that a main reason for the  obscurity of MEPs is that there are not many effective  methods for its computation and there is thus little to be gained from formulating a problem as an MEP. It is with this in mind that we propose a new homotopy method based on what we call \emph{fiber product homotopy} for computing MEP solutions.

We will now formally define an MEP in more conventional notations. Instead of having a single eigenvalue parameter $\lambda$, an MEP has multiple eigenvalue parameters $\Lam=(\lambda_{1},\dots,\lambda_{k})$. We will call
\[
H(\Lam) \coloneqq A_{0}-\lambda_{1}A_{1}-\dots-\lambda_{k}A_{k},
\]
a \emph{linear polynomial matrix} in $k$ parameters  $\lambda_{1},\dots,\lambda_{k}$ with matrix coefficients $A_0,\dots, A_k \in \mathbb{C}^{n \times n}$. We will write $\mathbb{P}^n$ for the complex projective $n$-space.
\begin{definition}\label{definition:MEP}
For a fixed $k \ge 2$ and given matrices $A_{ij}\in\mathbb{C}^{n_{i}\times n_{i}}$
with $j=0,1,\dots,k$,
$i=1,\dots,k$, consider the linear polynomial matrices
\[
H_{i}(\Lam)\coloneqq A_{i0}-\lambda_{1}A_{i1}-\lambda_{2}A_{i2}-\dots-\lambda_{k}A_{ik}.
\]
The \emph{multiparameter eigenvalue problem} (MEP), or, more precisely, a \emph{$k$-parameter eigenvalue problem}, is to find  $\lambda_{1},\dots,\lambda_{k} \in \mathbb{C}$
 and corresponding  $(\xv_{1},\dots,\xv_{k})\in\mathbb{P}^{n_{1}-1}\times\dots\times\mathbb{P}^{n_{k}-1}$
such that
\begin{equation}\label{eq:mep1}
H_{i}(\Lam)\xv_{i}=0, \qquad i=1,\dots,k.
\end{equation}
A solution $ (\lambda_{1},\dots,\lambda_{k},\xv_{1},\dots,\xv_{k}) \in \mathbb{C}^k \times \mathbb{P}^{n_{1}-1}\times\dots\times\mathbb{P}^{n_{k}-1}$ to the MEP is called an \emph{eigenpair},  the $k$-tuple
$(\xv_1,  \dots, \xv_k) \in  \mathbb{P}^{n_{1}-1}\times\dots\times\mathbb{P}^{n_{k}-1}$ an \emph{eigenvector},
 and the $k$-tuple $\Lam=(\lambda_{1},\dots,\lambda_{k}) \in \mathbb{C}^k$ an \emph{eigenvalue}.
 \end{definition}
Written out in full, \eqref{eq:mep1} takes the form
\begin{equation}\label{eq:mep2}
\begin{alignedat}{8}
(\lambda_1 A_{11}  &+  \lambda_2 A_{12}  &&+ \dots &&+ \lambda_k A_{1k} ) \xv_1  &&= A_{10} \xv_1, \\
(\lambda_1 A_{21}  &+  \lambda_2 A_{22}  &&+ \dots &&+ \lambda_k A_{2k} )\xv_2  &&= A_{20} \xv_2, \\
\MTFlushSpaceAbove
&\vdotswithin{+} && \vdotswithin{+}&& \vdotswithin{+}&& \vdotswithin{=}
\MTFlushSpaceBelow
(\lambda_1 A_{k1}  &+ \lambda_2 A_{k2}  &&+ \dots &&+ \lambda_k A_{kk}) \xv_k  &&=A_{k0} \xv_k.
\end{alignedat}
\end{equation}
With $\lambda_i$'s playing the role of $x_i$'s, $A_{ij}$'s and $A_{i0}$'s playing the roles of $a_{ij}$'s and $b_i$'s respectively  in \eqref{eq:lin}, and interpreting equality of the $i$th equation in \eqref{eq:lin}  to mean  equality on some $\xv_i \in \mathbb{P}^{n_i-1}$, we may view \eqref{eq:mep2} as an analogue of a linear system that we referred to at the beginning. The analogy is precise when $n_1 = \dots= n_k = 1$ --- \eqref{eq:mep2} is a linear system in the usual sense.

When $k = 1$, \eqref{eq:mep2} is a generalized eigenvalue problem. More generally, if $A_{ij} = 0$ for all $i \ne j$ \blue{and $j\neq 0$}, then \eqref{eq:mep2} is decoupled into $k$ generalized eigenvalue problems. Hence \eqref{eq:mep2} contains both eigenvalue problems and linear systems as special cases.
The multiparameter eigenvalue problem is well studied and readers may  refer to the books
\cite{Atkinson2011,Atkinson,Volkmer1988} for a comprehensive treatment.

Since any scalar multiple of $\xv_i$ is also an eigenvector  it is fitting to consider $\xv_i$ as an element of the projective space $\mathbb{P}^{n_i-1}$ although for practical reason one might prefer to simply normalize $\xv_i$ to have unit norm.

We develop a new homotopy method to solve a multiparameter eigenvalue problem effectively, where effectiveness is measured by the following factors:
\begin{itemize}
\item \emph{Speed} as measured by wall time.
We record time per path, maximum time over all paths, and total track time of all paths.
Our algorithm is highly parallelizable and the per-path times give good speed estimates when there are enough cores to track all paths in parallel.

\item  \emph{Accuracy} as measured by the backward error. We use the normwise backward error in   \cite[Theorem~2]{HP2003} for an approximate eigenpair. Our homotopy method tracks several copies of the eigenvalue $\Lam$; they should all converge to the same value if our method performs correctly and we include the difference between copies of $\Lam$'s as another measure of accuracy.

\item \emph{Certificates} of quadratic convergence in terms of Shub--Smale $\alpha$-theory.

\item  \emph{Number of divergent paths} that fail to converge to the solutions.
\end{itemize}
The last two measures only apply to  methods based on homotopy continuation.
We will compare our method to two existing methods:
\begin{enumerate}[\upshape (i)]
\item The \emph{Delta method} \cite{Atkinson}, which is the de facto standard method for solving MEPs by transforming them into a coupled system of generalized eigenvalue problems; we use the \textsc{MultiParEig} package \cite{MP2010} in our experiments with this method. \blue{For singular MEPs, we perform Delta method after extracting the common regular parts of the Delta matrices with a staircase algorithm.}
\item The \emph{diagonal coefficient homotopy method} recently proposed in \cite{Dong2016} for solving MEPs, where the start system is a random choice of diagonal matrices and the homotopy is a straight-line homotopy that deforms  \blue{$n_{1}+\cdots+n_{k}$} of \blue{$n_1+\cdots+ n_k+k$}
equations.
\end{enumerate}
The numerical experiments in  \cite{Dong2016}  show that the diagonal coefficient  homotopy method outperforms  the Delta method in terms of memory usage
and  is also faster for large $n_{1},\dots,n_{k}$.
Both methods find all eigenpairs of an MEP.

Our \emph{fiber product homotopy method} adopts an alternative approach --- we solve an MEP \eqref{eq:mep1} by solving a mathematically equivalent system that we will call the \emph{fiber product multiparameter eigenvalue problem}:\footnote{More precisely, we solve $H_{i}(\Lam_{i})\xv_{i}=0$, $G_1 = \dots = G_k = 0$, $i=1,\dots,k$, where each $G_i \colon \mathbb{C}^{k^2} \to \mathbb{C}^{k-1}$ is a linear function of $\Lam_{1},\dots,\Lam_{k}$, chosen so that the resulting system is equivalent to \eqref{eq:equivMEP}. See Section~\ref{sec:target}.}
\begin{equation}\label{eq:equivMEP}
H_{i}(\Lam_{i})\xv_{i}=0,\quad\Lam_{1}=\Lam_{2}=\dots=\Lam_{k},\qquad i=1,\dots,k,
\end{equation}
where $\Lam_{1},\dots,\Lam_{k}$ are to be regarded as different copies of $\Lam$.
Our corresponding homotopy
has a start system that
captures more structure of the MEP and allows us to deform far fewer equations.
By introducing $k(k-1)$ auxiliary unknowns, we deform at most $k(k-1)$
of the \blue{$n_1+\cdots+n_k+k^2$} equations.
\blue{For a fixed $k$ and $n_{1},\dots, n_{k} = O(n)$, fiber product homotopy deforms $O(1)$ equations whereas diagonal coefficient homotopy deforms $O(n)$ equations. Furthermore, fiber product homotopy deforms only \emph{linear} equations whereas diagonal coefficient homotopy deforms \emph{nonlinear} equations, which means that the paths of the fiber product homotopy are easier to track.
As we will see later, for \blue{dimension-deficient and} singular MEPs, the number of paths to track in a fiber product homotopy can be substantially smaller than using a diagonal coefficient homotopy.}

An eigenpair $(\Lam,\xv_1,\dots,\xv_k)$
of an MEP is said to be \emph{regular} if the eigenvalue $\Lam$ is isolated and has multiplicity one (see \cite{HP2003} for definitions of algebraic and geometric multiplicity).
Since the expected number of regular eigenpairs to an MEP is $n_{1}\cdots n_{k}$,
one often considers only relatively small values of $k$ to $n$ when finding all eigenpairs. Our fiber product homotopy method is guaranteed to compute all regular eigenpairs  in theory --- we show in  Theorem~\ref{thm:chosenCorrectly} that our start system in Section~\ref{sec:fibprodhom} is chosen correctly with probability one.

The fiber product homotopy method is motivated by several geometric insights. In Section~\ref{sec:eigen}, we define two algebraic varieties associated with a fiber product MEP: multiparameter eigenvalue variety and multiparameter eigenpair variety. \emph{Fiber products}, a notion well-known in areas from algebraic geometry to relational database, will be reviewed in Section~\ref{sec:fibprod}. The name for our method comes from the fact that a $k$-parameter eigenpair variety is a fiber product of $k$ one-parameter eigenpair varieties.  In Section~\ref{sec:condition}, we rely on geometry to define a condition number for the fiber product MEP \eqref{eq:equivMEP}, which differs from the condition number for the standard MEP \eqref{eq:mep2}, these being two different problems, albeit having the same solutions. Our condition number arises from intersecting an algebraic variety (our multiparameter eigenpair variety) with a varying linear space (defined by our start system).

We implemented a purely numerical version (in particular, it does not use multiprecision) of fiber product homotopy in \textsc{Matlab} and a mixed symbolic-numerical version in \textsc{Bertini} with \textsc{Macaulay2} for comparison. We did extensive numerical experiments with both implementations: randomly generated MEPs in Section~\ref{sec:rand};  the Mathieu two-parameter eigenvalue problem arising from a real-world application --- an elliptic membrane vibration problem --- in Section~\ref{sec:mathieu};  and \blue{MEPs that are both dimension-deficient and} singular, a challenging class of \blue{problems} with a deficiency in the number of eigenpairs and that breaks most other methods, in Section~\ref{sec:sing}.

We applied a broad range of measures for speed and accuracy to our numerical experiments to stress test the robustness of our method. In Section~\ref{sec:rand}, speed is measured via both wall time and iteration count (number of Newton steps); accuracy is measured via both relative backward error of the computed eigenpairs and the deviation in the multiple paths used to track the eigenvalues. In Section~\ref{sec:mathieu}, we test the effect of reducing the number of Newton steps (by early stopping) on the accuracy of our method and certify its final quadratic convergence speed using Shub--Smale $\alpha$-theory. For the \blue{dimension-deficient} singular MEPs in Section~\ref{sec:sing}, what breaks other homotopy methods is the issue of divergent paths and so we use the number of divergent paths as a measure of effectiveness. The take-away is that our method saw \emph{zero} divergent path in every case  we tested. In Sections~\ref{sec:rand} and \ref{sec:sing}, we also provide the time it takes to track a single path, which gives a good estimate of the speed under parallel execution of our method (since each path can be tracked independently of~others).

\section{Homotopy methods}\label{sec:homotopy}

We recall the straight-line homotopy and describe the diagonal coefficient homotopy method used in \cite{Dong2016} but will defer the description of our fiber product homotopy method to Section~\ref{sec:fibprodhom}.

A homotopy deforms solutions of a
\emph{start system} $\boldsymbol{Q}(\zv)=0$
to solutions of a
\emph{target system} $\boldsymbol{P}(\zv)=0$.
More precisely,
a \emph{straight-line homotopy}
with {path parameter} $t$ is defined as
\begin{equation}\label{eq:H}
\boldsymbol{H}(\zv,t)\coloneqq (1-t)\boldsymbol{Q}(\zv)+ t\boldsymbol{P}(\zv), \quad t\in[0,1].
\end{equation}
When $t=0$, $\boldsymbol{H}(\zv,0)=\boldsymbol{Q}(\zv)=0$ is the start system and when $t=1$, $\boldsymbol{H}(\zv,1)=\boldsymbol{P}(\zv)=0$ is the target system.

\begin{definition}\label{def:correctly}
A start system for the homotopy $\eqref{eq:H}$
is said to be \emph{chosen correctly}  \cite{li_1997} if the
 following properties hold:
\begin{enumerate}[\upshape (i)]
\item the solution set of the start system $\boldsymbol{Q}(\zv)=0$ are known or easy to obtain;
\item\label{item:wcSmooth} the solution set of $\boldsymbol{H}(\zv,t)=0$ for $0 \leq t < 1$ consists of a finite number of smooth paths, each parametrized by $t$ in $[0,1)$;
\item for each isolated solution of the target system $\boldsymbol{P}(\zv)=0$, there is some path originating from a solution of the start system that leads to it.
\end{enumerate}
\end{definition}

Let $D_{ij}$ denote diagonal matrices in $\mathbb{C}^{n_i\times n_i}$.
The \emph{diagonal coefficient homotopy method} for solving MEP
is the straight-line homotopy given by:
\begin{equation}\label{eq:diagPQ}
\begin{aligned}
\boldsymbol{H}_{\DC}(\Lam,\xv_1,\dots,\xv_k,t)
&\coloneqq (1-t)\boldsymbol{Q}_{\DC}+ t\boldsymbol{P}_{\DC},\\
\boldsymbol{Q}_{\DC}(\Lam,\xv_{1},\dots,\xv_{k})
&\coloneqq \begin{bmatrix}
(D_{10}+\lambda_1D_{11})\xv_1 \\
\vdots\\
(D_{k0}+\lambda_k D_{kk})\xv_k
\end{bmatrix},
 \\
\boldsymbol{P}_{\DC}(\Lam,\xv_{1},\dots,\xv_{k})
&\coloneqq \begin{bmatrix}
H_{1}(\Lam)\xv_{1} \\
\vdots\\
H_{k}(\Lam)\xv_{k} \\
\end{bmatrix}
\end{aligned}
\end{equation}
where one regards
\begin{align*}
\boldsymbol{H}_{\DC} \colon \mathbb{C}^k\times(\mathbb{P}^{n_1-1}\times\dots\times\mathbb{P}^{n_k-1})\times\mathbb{C} & \to\mathbb{C}^{n_1+\dots +n_k},\\
\boldsymbol{Q}_{\DC}, \; \boldsymbol{P}_{\DC} \colon \mathbb{C}^k\times(\mathbb{P}^{n_1-1}\times\dots\times\mathbb{P}^{n_k-1})&\to\mathbb{C}^{n_1+\dots +n_k}
\end{align*}
as polynomial maps.
The homotopy method proposed in \cite{Dong2016} is an example of a diagonal coefficient homotopy method.

To implement the homotopy above, one has to account for the scaling of the eigenvectors by introducing $k$ constraints.
One way to do this is by scaling each $\xv_i$ so that it has norm one.
This is the approach undertaken in \cite{Dong2016}.
Another way to do this is to place a generic\footnote{Here ``generic'' is used in its usual sense in algebraic geometry. Those unfamiliar with this notion may assume that it is synonymous with ``random.''} affine constraint on each $\xv_i$, which is what we will do in our fiber product homotopy in Section~\ref{sec:fibprodhom}\blue{.}

\section{Multiparameter eigenvarieties}\label{sec:eigen}

We will define two algebraic varieties associated with an MEP.
\begin{definition}\label{def:epvar}
In algebraic geometric terms \cite{Hartshorne}, the coordinates of  $H(\Lam)\xv$  are polynomials that form a subset of $\mathbb{C}[\Lam,\xv] =\mathbb{C}[\lambda_1,\dots,\lambda_k,x_1,\dots,x_n]$ and define an algebraic variety
\[
\EP(H)\coloneqq \{
(\Lam,\xv) \in \mathbb{C}^k\times\mathbb{P}^{n-1} :  H(\Lam)\xv=0 \}.
\]
We will call this the \emph{eigenpair variety} of $H$.
In the context of an MEP, we will call
the Cartesian product $\EP(H_1,\dots,H_k)\coloneqq \EP(H_1)\times\dots\times \EP(H_k)$
the
\emph{multiparameter eigenpair variety} of $H_1,\dots,H_k$. Explicitly,
\begin{multline}\label{eq:defZ}
\EP(H_1,\dots,H_k)=
\{
(\Lam_1,\dots\Lam_k,\xv_1,\dots,\xv_k) \in \mathbb{C}^{k^2} \times \mathbb{P}^{n_1-1}\times\dots\times\mathbb{P}^{n_k-1} : \\
H_1(\Lam_1)\xv_1=0,\dots,H_k(\Lam_k)\xv_k=0 \}.
\end{multline}
The \emph{multiparameter eigenvalue variety} of $H_1,\dots,H_k$ is the coordinatewise projection of the multiparameter eigenpair variety to $\mathbb{C}^{k^2}$ and will be denoted by $\EV\bigl(H_1,\dots,H_k\bigr)$. Alternatively, it can be defined explicitly as
\begin{multline}\label{eq:EVV}
\EV(H_1,\dots,H_k)=
\{ (\Lam_1,\dots\Lam_k) \in \mathbb{C}^{k^2}  :
\det H_1(\Lam_1)=0,\dots,\det H_k(\Lam_k)=0\}.
\end{multline}
\end{definition}

\blue{Througout our article, $\mathbb{N}$ will denote the set of \emph{positive} integers.}
\begin{definition}\label{definition:intrinsicDimension}
An MEP is said to have \emph{intrinsic dimension} \blue{$(d_1,\dots,d_k)\in\mathbb{N}^k$}
if the \blue{total} degree of the polynomial $\det H_i(\Lam_i)$ is $d_i$, $i=1,\dots,k$.
Such an MEP is said to be \emph{generic with respect to intrinsic dimension} if
the hypersurface defined by $\det H_i(\Lam_i)$ is generically reduced\footnote{This is an algebraic geometry term that implies the polynomial $\det H_i(\Lam_i)$
 is square-free.} for $i=1,\dots,k$.
\end{definition}

\blue{As $d_i$'s are required to be positive, none of the $\det H_i(\Lam_i)$'s are the zero polynomial.}
By Bezout's theorem, the degree of the multiparameter eigenvalue variety $\EV(H_1,\dots,H_k)\subseteq \mathbb{C}^{k^2}$ is at most $\prod_{i=1}^k d_i$.
So  the number of isolated regular points in the intersection of $\EV(H_1,\dots,H_k)$ with a codimension-$k(k-1)$
affine linear space in $\mathbb{C}^{k^2}$
is at most $\prod_{i=1}^k d_i$.  We state this formally below.
\begin{proposition}\label{prop:N}
An MEP with intrinsic dimension $(d_1,\dots,d_k)$ has at
 most  $\prod_{i=1}^kd_i$ regular eigenvalues and eigenpairs.
This bound is tight if the MEP is generic with respect to intrinsic dimension.
 \end{proposition}

The next  observation will be a key to our fiber product homotopy method.
\begin{lemma}\label{lemma:correspond}
Let $\mathbb{D}$ denote the linear space defined by
\[
\mathbb{D} \coloneqq \{ (\Lam_1,\dots\Lam_k,\xv_1,\dots,\xv_k) \in \mathbb{C}^{k^2} \times \mathbb{P}^{n_1-1}\times\dots\times\mathbb{P}^{n_k-1} :  \Lam_1=\dots=\Lam_k\}.
\]
Then the intersection $\mathbb{D} \cap \EP(H_1,\dots,H_k)$ gives the set of eigenpairs of \eqref{eq:mep1}.
\end{lemma}
In fact, $\mathbb{D}\cap\EP(H_1,\dots,H_k)$ is the \emph{fiber product} of $\EP(H_1),\dots,\EP(H_k)$ over $\mathbb{C}^k$, a standard notion in algebraic geometry \cite{Hartshorne}. This is the impetus for the name of our homotopy method --- fiber product homotopy.

\section{Fiber products}\label{sec:fibprod}

Knowledge of the fiber product's formal properties at the level of, say, \cite{Hartshorne} is unnecessary for us. All we need is the notion of  fiber product of sets --- an important and well-known concept in relational database theory \cite[Section~6.3]{CMCS1}.

Let $X, Y, A$ be sets and $\varphi \colon X  \to A$ and $\psi \colon Y \to A$ be maps.
\[
\begin{tikzpicture}
  \matrix (m) [matrix of math nodes,row sep=3em,column sep=4em,minimum width=2em]
{      & Y \\
     X & A \\ };
  \path[-stealth]
    (m-2-1.east|-m-2-2) edge node [below] {$\varphi$}
            node [above] {} (m-2-2)
    (m-1-2) edge node [right] {$\psi$} (m-2-2);
\end{tikzpicture}
\]
The \emph{fiber product}
$X\times_A Y$ of  $X$ and $Y$ over $A$
is the subset of
$X\times Y$ given by
\[
X \times_A Y \coloneqq \{
(x,y) \in X \times Y \colon
\varphi(x)=\psi(y)
\}.
\]
\blue{Note that a fiber product depends on the maps $\varphi$ and $\psi$ although this is not reflected in the notation $X \times_A Y$, which is nevertheless standard.}
The fiber product satisfies the following commutative diagram where $\pi_1(x,y)=x$ and $\pi_2(x,y)=y$ are projection maps:
\[
\begin{tikzpicture}
  \matrix (m) [matrix of math nodes,row sep=3em,column sep=4em,minimum width=2em]
  {
     X \times_A Y & Y \\
     X & A\\ };
  \path[-stealth]
    (m-1-1) edge node [left] {$\pi_1$} (m-2-1)
            edge node [below] {$\pi_2$} (m-1-2)
    (m-2-1.east|-m-2-2) edge node [below] {$\varphi$}
            node [above] {} (m-2-2)
    (m-1-2) edge node [right] {$\psi$} (m-2-2);
\end{tikzpicture}
\]
We will illustrate fiber products with a few examples.

\begin{example}[Relational database]
Let
 $X = \{1,2,3,4\}$,
 $Y = \{a,b,c,d,e\}$, and $A=\{-1,+1\}$.
Let the maps  $\varphi \colon X \to A$ and $\psi \colon Y \to A$ be given by
\begin{equation*}
\varphi(x) =
     \begin{cases}
       +1 &\text{$x$ is odd},\\
       -1 &\text{$x$ is even},
     \end{cases}
\qquad
\psi(y) =
     \begin{cases}
       +1 &\text{$y$ is a vowel},\\
       -1 &\text{$y$ is a consonant}.
     \end{cases}
\end{equation*}
Then the fiber product of $X$ and $Y$ over $A$ is
\[
X \times_A Y = \{
(1,a),
(1,e),
(2,b),
(2,c),
(2,d),
(3,a),
(3,e),
(4,b),
(4,c),
(4,d)\}.
\]
Incidentally, this toy example underlies the \texttt{JOIN} operation in the structured query language (\textsc{sql}) of a relational database management system (\textsc{rdbms}). See  \cite[Section~6.3]{CMCS1} for more information.
\end{example}

\begin{example}[Algebraic geometry]
Consider the following cubic curves in $\mathbb{R}^2$:
 \begin{align*}
 X &= \{(t_1,z_1)\in \mathbb{R}^2  : t_1=z_1(z_1-1)(z_1-2)+1\},\\
 Y &= \{(t_2,z_2)\in \mathbb{R}^2 : t_2=z_2(z_2-1)(z_2-2)+1\}.
\end{align*}
Let $A=\mathbb{R}$ and consider the maps
\[
\varphi \colon X \to A,\quad
 \varphi(t_1,z_1) = t_1 \qquad\text{and}\qquad
\psi \colon Y \to A,\quad
 \psi(t_2,z_2) = t_2.
\]
Their fiber product,
\[
X \times_A Y = \{
(t_1,t_2,z_1,z_2)\in\mathbb{R}^4 :
(t_1,z_1)\in X,\;
(t_2,z_2)\in Y,\; t_1=t_2
\},
\]
is shown in Figure~\ref{fig:fpExampleCubic}.
While the Cartesian product $X\times Y$ is an irreducible surface,
the fiber product $X \times_A Y$ is a union of two curves ---
a point $(t_1,t_2,z_1,z_2)\in X \times_A Y$ satisfies
\[
z_1=z_2\qquad  \text{or}  \qquad
z_1^2+z_1z_2+z_2^2=3z_1+3z_2-2.\]
One of the curves projects onto a line and the other projects onto an ellipse.
Whereas the Cartesian product of two irreducible curves is always an irreducible surface,
this example shows that the fiber product of two irreducible curves does not need to be irreducible.

\begin{figure}[htb!]
\centering
\caption{$X \times_A Y$ and $X\times Y$ are subsets of $\mathbb{R}^4$, but we only plot the
coordinates $(z_1,z_2,t_1)$.
\emph{Left:} The fiber product $X \times_A Y$ is the union of the blue curve and the magenta curve.
The dotted blue line and dotted magenta ellipse are the $(z_1,z_2)$-coordinate projections of the corresponding curve onto the $t_1=0$ plane.
\emph{Right:} The $(z_1,z_2,t_1)$-coordinate projection of the Cartesian product $X\times Y$.
}
\includegraphics[trim={21ex, 13ex, 12ex, 13ex}, clip, width=0.49\textwidth]{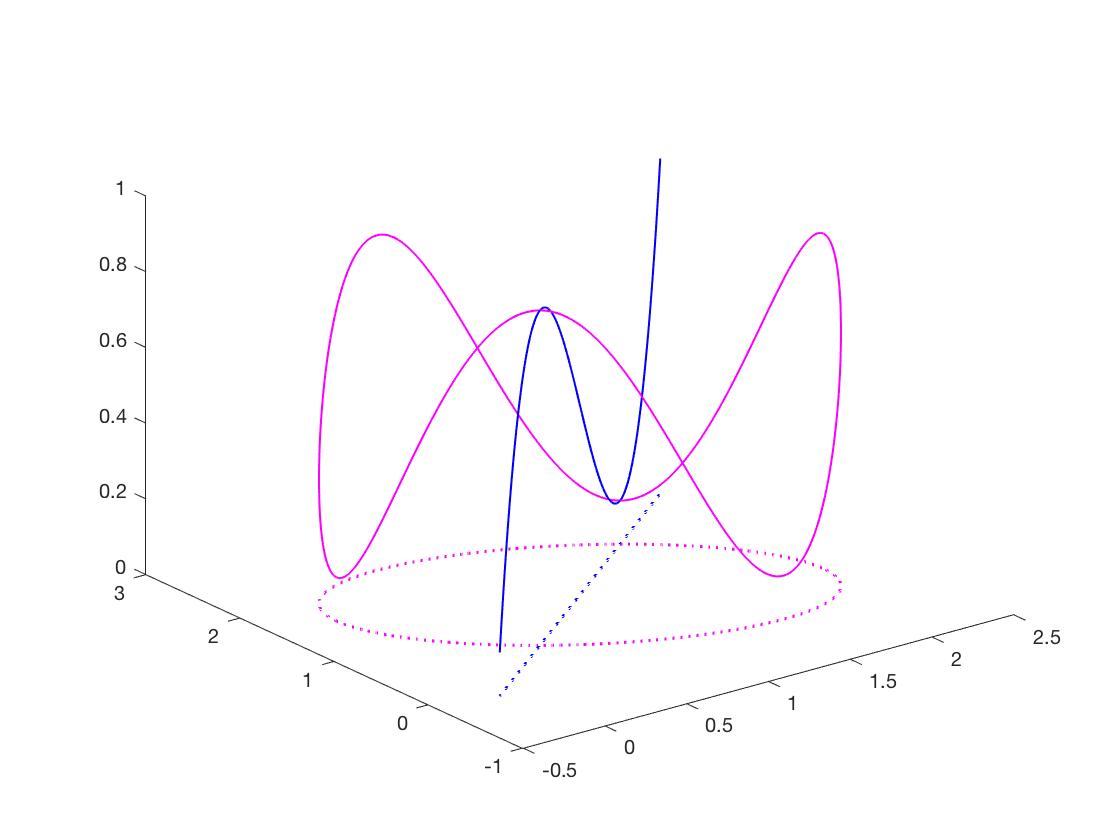}
\includegraphics[trim={21ex, 13ex, 12ex, 13ex}, clip, width=0.49\textwidth]{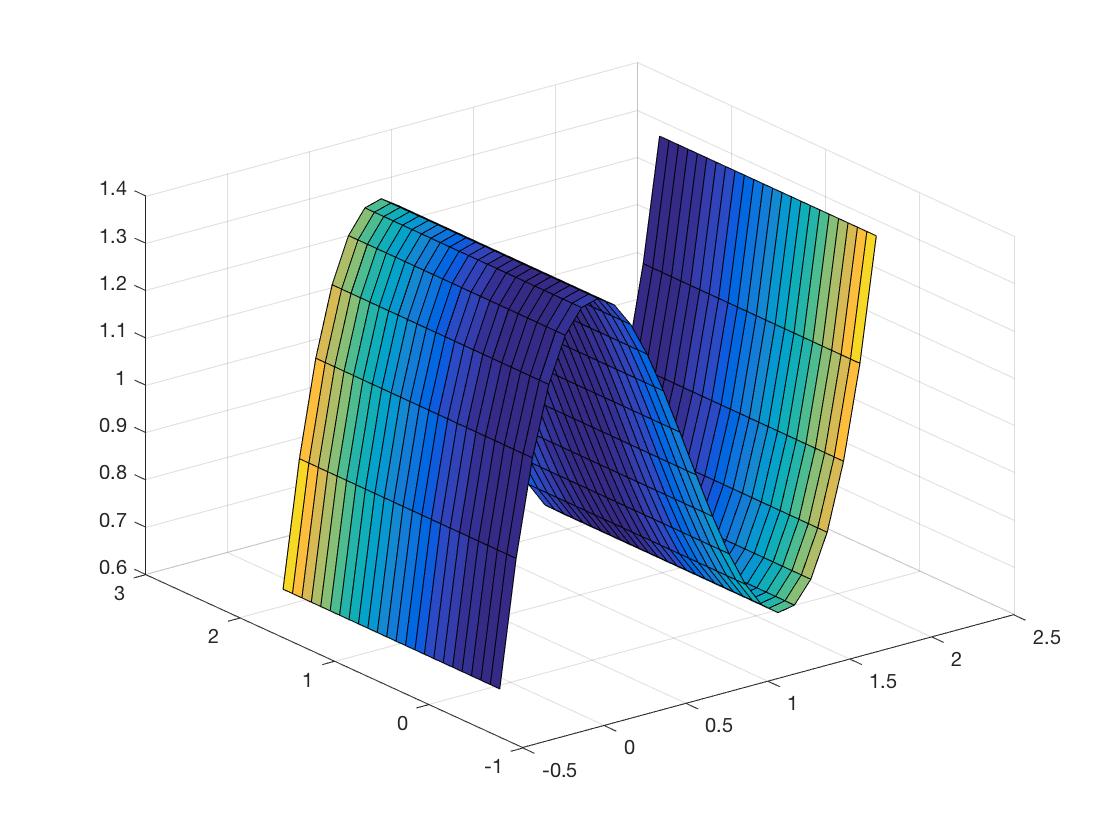}
\label{fig:fpExampleCubic}
\end{figure}

\end{example}

\begin{example}[Multiparameter eigenvalue problem]
Consider a two-parameter eigenvalue problem \eqref{eq:mep1} 
and the eigenpair varieties  of $H_1$ and $H_2$,
\begin{align*}
 X &= \{(\Lam_1,\xv_1)\in \mathbb{C}^2\times\mathbb{P}^{n_1-1} : H_1(\Lam_1)\xv_1=0\},\\
 Y &= \{(\Lam_2,\xv_2)\in \mathbb{C}^2\times\mathbb{P}^{n_2-1} : H_2(\Lam_2)\xv_2=0\}.
 \end{align*}
Let  $A=\mathbb{C}^2$ and consider the maps
\[
\varphi \colon X \to A, \quad
\varphi(\Lam_1,\xv_1) = \Lam_1\qquad\text{and}\qquad
\psi \colon Y \to A, \quad
 \psi(\Lam_2,\xv_2) = \Lam_2.
\]
Their fiber product is
\begin{multline*}
X \times_A Y = \{
(\Lam_1,\Lam_2,\xv_1,\xv_2) \in \mathbb{C}^{4} \times \mathbb{P}^{n_1-1}\times \mathbb{P}^{n_2-1} : \\
(\Lam_1,\xv_1)\in X,\;
(\Lam_2,\xv_2)\in Y,\; \Lam_1=\Lam_2
\}
= \EP(H_1,H_2),
\end{multline*}
the two-parameter eigenpair variety of $H_1, H_2$.
\end{example}

\section{Fiber product homotopy method}\label{sec:fibprodhom}

Fiber products have previously appeared in numerical algebraic geometry in various contexts: study of exceptional sets \cite{SW2008}, algorithms to intersect varieties \cite{HW2017}, and numerical computations of Galois groups \cite[Section~4]{HRS2017}. However, the use of fiber products in a homotopy method for solving MEPs is, as far as we know, new. We will now describe this method.

\subsection{Start system}

We choose our start system by the following observation.
Let  $H(\Lam) = A_{0}-\lambda_{1}A_{1}-\dots-\lambda_{k}A_{k}$ be a linear polynomial matrix and $L(\Lam)= A \Lam - \boldsymbol{b}$ be an affine linear function for some $A \in \mathbb{C}^{(k-1)\times k}$ and $\boldsymbol{b} \in \mathbb{C}^{k-1}$. We claim that
\[
H(\Lam)\xv=0,\quad L(\Lam)=0
\]
is equivalent to a generalized eigenvalue problem, that we will call the \emph{associated generalized eigenvalue problem} or associated GEP for short.

To see this, let $\Lam =\beta \qv+\pv$ where  $\pv, \qv \in \mathbb{C}^k$ are such that $L(\pv) = 0$ and $\qv\in\ker(\Jac L)$, where $\Jac$ denotes the total derivative (also known as total differential). Note that the previous statement says nothing more than $A \pv  = \boldsymbol{b}$ and $A \qv = 0$.
Eliminating $\qv$ and $\pv$ from
\begin{equation}\label{eq:restrictGEP}
H(\beta \qv+\pv)\xv=0,  \quad \qv\in\ker(\Jac L),  \quad  L(\pv)=0,
\end{equation}
then gives us a GEP with $\beta$ the generalized eigenvalue and $\xv$ the generalized eigenvector. We will see in Example~\ref{eg:gep} how one may obtain a GEP from \eqref{eq:restrictGEP} but expressing the GEP in terms of general $A$, $\boldsymbol{b}$, and $A_0, \dots, A_k$ is complicated and unilluminating.
The \emph{associated eigenpairs} are
\begin{equation}\label{eq:associatedEP}
\{ (\Lam,\xv)  \in \mathbb{C}^k\times\mathbb{P}^{n-1} :  \Lam=\beta \qv+\pv, \;
 H(\beta \qv+\pv)\xv=0 \}.
\end{equation}
As each coordinate of $ H(\beta \qv+\pv)\xv $ is homogeneous  in $\xv$, $ H(\beta \qv+\pv)\xv =0$  is well-defined for $\xv  \in \mathbb{P}^{n-1}$.

For an MEP \eqref{eq:mep1}, let $L_i \colon \mathbb{C}^k \to \mathbb{C}^{k-1}$, $i=1,\dots,k$, be affine linear maps. Thus each $L_i(\Lam_i) =0 $ is an affine linear equation in $\Lam_{i}$, $i=1,\dots,k$. We obtain $k$ associated GEPs:
\begin{equation}\label{eq:solveProduct}
H_i(\Lam_i)\xv_i = 0,\quad L_i(\Lam_i)=0, \quad i=1,\dots,k.
\end{equation}

The set of solutions to \eqref{eq:solveProduct} will be called \emph{start solutions} or \emph{start points} and denoted $S$.
 If the MEP has intrinsic dimension $(d_1,\dots,d_k)$,
then the $i$th associated GEP has at most $d_i$ generalized eigenvalues.
Thus if $S_{i}$ denotes the set of associated eigenpairs that have distinct generalized eigenvalues,
then $S$ is given by the Cartesian product
\[
S =  S_{1}\times\dots\times S_{k}
\]
and it has cardinality $\prod_{i=1}^k \lvert S_i\rvert \le \prod_{i=1}^kd_i$.

We define $\boldsymbol{Q}_\FP \colon \mathbb{C}^{k^2} \times \mathbb{P}^{n_1-1}\times\dots\times\mathbb{P}^{n_k-1} \to
\mathbb{C}^{n_1 + \dots + n_k} \times \mathbb{C}^{k(k-1)}$ by
\begin{align}\label{eq:startSystem}
\boldsymbol{Q}_\FP(\Lam_{1},\dots,\Lam_k,\xv_{1},\dots,\xv_{k})&
\coloneqq \begin{bmatrix}
H_{1}(\Lam_{1})\xv_{1}\\
\vdots\\
H_{k}(\Lam_{k})\xv_{k}\\
L_{1}(\Lam_{1})\\
\vdots\\
L_{k}(\Lam_{k})
\end{bmatrix}
\end{align}
and choose our \emph{start system} to be the $\sum_{i=1}^k(k-1+n_i)$ equations in  $\sum_{i=1}^k(k+n_i)$ variables
\begin{equation}\label{eq:ss=0}
\boldsymbol{Q}_\FP(\Lam_{1},\dots,\Lam_k,\xv_{1},\dots,\xv_{k})=0.
\end{equation}
Again note that $H_{i}(\Lam_{i})\xv_{i} =0$  is well-defined for $\xv_i  \in \mathbb{P}^{n_i-1}$, $i =1,\dots, k$. We have the following easy observation.
\begin{lemma}\label{lemma:startSystem}
The points in $S$
are  regular solutions to
the start system \eqref{eq:ss=0}.
If the MEP is generic with respect to intrinsic dimension, then
$\lvert S\rvert=\prod_{i=1}^kd_i$ for any generic choice of affine linear maps $L_1,\dots,L_k$.
\end{lemma}

We now give an illustration of how an MEP can be transformed into a system of GEPs by imposing random affine constraints.
\begin{example}\label{eg:gep}
Consider the two-parameter eigenvalue problem given by
\begin{align*}
H_1(\Lam_{1})&=
\begin{bmatrix}
2 & 3\\
5 & 7
\end{bmatrix}+\lambda_{11}\begin{bmatrix}
11 & 13\\
17 & 19
\end{bmatrix}+\lambda_{12}\begin{bmatrix}
23 & 29\\
31 & 37
\end{bmatrix},\\
H_2(\Lam_{2})&=
\begin{bmatrix}
12 & 31\\
15 & 71
\end{bmatrix}+\lambda_{21}\begin{bmatrix}
1 & 1\\
1 & 1
\end{bmatrix}+\lambda_{22}\begin{bmatrix}
2 & 2\\
2 & 2
\end{bmatrix},
\end{align*}
where we write $\Lam_i = (\lambda_{i1}, \lambda_{i2})$, $i =1,2$. \blue{This two-parameter eigenvalue problem is singular and has only two regular eigenvalues.}

We pick two random affine linear polynomials $L_1$ and $L_2$, e.g.,
\begin{align*}
L_1(\Lam_1)&=
(0.6909 + 0.2745\ii)\lambda_{11}+
(0.4277 - 0.1333\ii)\lambda_{12}-1,
\\
L_2(\Lam_2)&=
(-0.1443 + 0.5711\ii)\lambda_{21}+
(-0.0735 + 1.8085\ii)\lambda_{22}
-1.
\end{align*}
Then we have
\[\small
[\qv_{1}, \qv_{2},
\pv_{1}, \pv_{2}]=
\begin{bmatrix}
-0.41+0.31\ii & -0.52+0.79\ii &1.2501-0.4967\ii & 0\\
0.85+0.02\ii & 0.22-0.21\ii &0 & -0.0224-0.5520\ii
\end{bmatrix}.
\]
The polynomial matrices for the two associated GEP are {\small
\begin{align*}
H_1(\beta \qv_1+\pv_1) &=
\begin{bmatrix}
 -11.75 + 5.46\ii& -13.25 + 6.45\ii\\
 -16.25 + 8.44\ii& -16.75 + 9.43\ii
\end{bmatrix}
 -\beta
\begin{bmatrix}
  15.17 + 4.02\ii&  19.48 + 4.79\ii\\
  19.55 + 6.10\ii&  23.87 + 6.87\ii
\end{bmatrix},
\\
H_2(\beta \qv_2+\pv_2)&=
\begin{bmatrix}
  12.04 + 1.10\ii&  31.04 + 1.10\ii\\
  15.04 + 1.10\ii&  71.04 + 1.10\ii
\end{bmatrix}
-\beta
\begin{bmatrix}
  -0.08 + 0.35\ii&  -0.08 + 0.35\ii\\
  -0.08 + 0.35\ii & -0.08 + 0.35\ii
\end{bmatrix}.
\end{align*}}%
The first GEP has two finite eigenvalues:
 $ -0.9978 + 1.1933\ii$ and
$ -0.5637 + 0.3035\ii$.
The second  GEP has only one finite eigenvalue at
$-3.6333 -28.4804\ii$. \blue{Note that the values $d_1 = 2$, $d_2 = 1$ are in agreement with the number of regular eigenvalues.}
Thus the \emph{start solutions}, i.e., the solutions to our start system, of our homotopy is a set of two points with
$(\Lam_1,\Lam_2)$-coordinates below.
\[
\begin{array}{c|c}
\Lam_1 & \Lam_2\\ \hline
\\[-2ex]
( -0.9978 + 1.1933\ii)\qv_1+\pv_1
&(-3.6333 -28.4804\ii)\qv_2+\pv_2 \\[0.75ex]
(  -0.5637 + 0.3035\ii)\qv_1+\pv_1
&(-3.6333 -28.4804\ii)\qv_2+\pv_2
\end{array}
\]

\end{example}

\subsection{Target system}\label{sec:target}

For $i =1,\dots, k$, let $R_i \in \mathbb{C}^{(k-1)\times k(k-1)}$ be generic and let $G_i \colon \mathbb{C}^{k^2} \to \mathbb{C}^{k-1}$ be the linear function defined by
\begin{equation}\label{eq:defineG}
G_i(\Lam_1,\dots,\Lam_k)\coloneqq R_i D_k
\begin{bmatrix}
\Lam_{1}\\
\vdots\\
\Lam_{k}
\end{bmatrix},
\;
D_k\coloneqq
\begin{bmatrix}
I_{k} & -I_{k}\\
  &  \ddots & \ddots\\
  &   & I_{k} & -I_{k}
\end{bmatrix} \in \mathbb{C}^{k(k-1) \times k^2},
\end{equation}
where $I_k$ is the $k\times k$ identity matrix.
If  the matrices $R_1,\dots,R_k \in \mathbb{C}^{(k-1)\times k(k-1)}$ are generic, then
\begin{equation}\label{eq:G=}
\begin{aligned}
\{ (\Lam_1,\dots\Lam_k) &\in \mathbb{C}^{k^2} :
G_1(\Lam_1,\dots,\Lam_k)=\dots=G_k(\Lam_1,\dots,\Lam_k)=0 \} \\
=\{  (\Lam_1,\dots\Lam_k) &\in \mathbb{C}^{k^2} :  \Lam_1=\dots=\Lam_k \},
\end{aligned}
\end{equation}
and therefore,
\begin{align*}
\{ (\Lam_1,\dots\Lam_k,\xv_1,\dots,\xv_k) &\in \mathbb{C}^{k^2}\times\mathbb{P}^{n_1-1}\times\cdots\times\mathbb{P}^{n_k-1}  :
 G_1(\Lam_1,\dots,\Lam_k)=\dots=G_k(\Lam_1,\dots,\Lam_k)=0 \} \\
=\{  (\Lam_1,\dots\Lam_k,\xv_1,\dots,\xv_k) &\in \mathbb{C}^{k^2}\times\mathbb{P}^{n_1-1}\times\cdots\times\mathbb{P}^{n_k-1} : \Lam_1=\dots=\Lam_k \}=\mathbb{D},
\end{align*}
i.e., the linear space in Lemma~\ref{lemma:correspond}. Hence the system
\begin{equation}\label{eq:mep3}
H_{i}(\Lam_{i})\xv_{i}=0,\quad G_1(\Lam_1,\dots,\Lam_k)=\dots=G_k(\Lam_1,\dots,\Lam_k)=0,\qquad i=1,\dots,k,
\end{equation}
is equivalent to the fiber product MEP \eqref{eq:equivMEP}, which is  equivalent to the original MEP in \eqref{eq:mep1}.
We define
\begin{align}\label{eq:targetSystem}
\boldsymbol{P}_\FP(\Lam_{1},\dots,\Lam_k,\xv_{1},\dots,\xv_{k})
&\coloneqq \begin{bmatrix}
H_{1}(\Lam_{1})\xv_{1}\\
\vdots\\
H_{k}(\Lam_{k})\xv_{k}\\
G_{1}(\Lam_{1},\dots,\Lam_{k})\\
\vdots\\
G_{k}(\Lam_{1},\dots,\Lam_{k})
\end{bmatrix},
\end{align}
and choose our \emph{target system} to be
\[
\boldsymbol{P}_\FP(\Lam_{1},\dots,\Lam_k,\xv_{1},\dots,\xv_{k})=0,
\]
which is of course just \eqref{eq:mep3}.
 By Proposition~\ref{lemma:correspond},
our target system \eqref{eq:targetSystem} yields the eigenpairs of the MEP.

\subsection{Fiber product homotopy}\label{ss:fph}

The main objective of this section is to show that our start system is chosen correctly with probability one.
\begin{definition}\label{defn:ourHomotopy}
A \emph{fiber product homotopy} for the MEP with linear polynomial matrices
$H_1,\dots,H_k$  is a straight-line homotopy from $t=0$ to $t=1$ given by the polynomial map
\begin{gather}
\boldsymbol{H}_\FP \colon \mathbb{C}^{k^2}\times(\mathbb{P}^{n_1-1}\times\dots\times\mathbb{P}^{n_k-1})\times\mathbb{C}\to\mathbb{C}^{n_1+\dots +n_k}\times\mathbb{C}^{(k-1)k}, \notag\\
\boldsymbol{H}_\FP(\Lam_1,\dots,\Lam_k,\xv_1,\dots,\xv_k,t)
\coloneqq\begin{bmatrix}
H_{1}(\Lam_{1})\xv_{1}\\
\vdots\\
H_{k}(\Lam_{k})\xv_{k}\\
(1-t)L_{1}(\Lam_{1})+tG_{1}(\Lam_{1},\dots,\Lam_{k})\\
\vdots\\
(1-t)L_{k}(\Lam_{k})+tG_{k}(\Lam_{1},\dots,\Lam_{k})
\end{bmatrix}.  \label{eq:homotopyPQ}
\end{gather}
Note that $\boldsymbol{H}_\FP = (1-t)\boldsymbol{Q}_\FP+t\boldsymbol{P}_\FP$.
\end{definition}
Throughout the article we assume that $k \ge 2$ so that we indeed have a multiparameter eigenvalue problem. If $k=1$, then $k(k-1)=0$, and \eqref{eq:homotopyPQ} will not involve the path parameter $t$. If the reader is wondering whether our homotopy method applies to a standard  or generalized eigenvalue problem like in \cite{Zhang-Law-Golub}, this shows that the answer is no.

\begin{theorem}\label{thm:chosenCorrectly}
The fiber product homotopy for an MEP \eqref{eq:homotopyPQ}
with intrinsic dimension $(d_1,\dots,d_k)$
has a start system chosen correctly with probability one
if the start solutions $S$  has $\lvert S\rvert=\prod_{i=1}^kd_i$.
\end{theorem}
\begin{proof}
By Lemma~\ref{lemma:startSystem},
the start solutions $S$
are known after solving $k$ generalized eigenvalue problems of dimensions $n_1,\dots,n_k$  respectively.
Consider the variety
\begin{multline*}
\mathcal{C}\coloneqq
\{
(\Lam_1,\dots,\Lam_k,\xv_1,\dots,\xv_k)\in\mathbb{C}^{k^2}\times \mathbb{P}^{n_1-1}\times\cdots\times\mathbb{P}^{n_k-1}
\colon
L(\Lam_1)=0,\dots,L(\Lam_k)=0
\}.
\end{multline*}
The intersection $\mathcal{C}\cap\EP(H_1,\dots,H_k)$ consists of $\lvert S \rvert=\prod_{i=1}^kd_i$ points.
As each $G_i$ is a linear function and each $R_i \in \mathbb{C}^{(k-1)\times k(k-1)}$ is generic, $i=1,\dots,k$,
it follows from the gamma trick \cite[Lemma~7.1.3]{SW05}
that for regular eigenpairs, the homotopy has a start system chosen correctly with probability one.
\end{proof}

We will provide more extensive numerical experiments in Section~\ref{sec:implement} but here we illustrate  our method with a small example: $k=3$ and $n_{1}=n_{2}=n_{3}=2$.
\begin{example}\label{eg:3222}
We generate an MEP by randomly choosing the $2\times 2$ coefficient matrices $A_{ij}$ for $i=1,2,3$ and $j=0,1,2,3$.  There are eight solutions to the start and target systems.
Note that fiber homotopy method requires that we work \blue{with $(\Lam_1,\Lam_2,\Lam_3,\xv_1,\xv_2,\xv_3) \in \mathbb{C}^{9} \times \mathbb{P}^1\times \mathbb{P}^1\times \mathbb{P}^1$}. The end point will however be of the form  $(\Lam,\Lam,\Lam,\xv_1,\xv_2,\xv_3)\in \blue{\mathbb{C}^{9} \times \mathbb{P}^1\times \mathbb{P}^1\times \mathbb{P}^1}$ where  $(\Lam, \xv_1,\xv_2,\xv_3) \in \blue{\mathbb{C}^{3} \times \mathbb{P}^1\times \mathbb{P}^1\times \mathbb{P}^1}$ is a \blue{multiparameter} eigenpair.

With our homotopy \eqref{eq:homotopyPQ}, we deform from $t=0$ to $t=1$.  Note that $\Lam_i=(\lambda_{i1},\lambda_{i2},\lambda_{i3}) \in \mathbb{C}^3$, $i=1,2,3$. In the left plot of Figure~\ref{eightEndPaths}, we track the $\lambda_{i1}$-coordinate of all eights paths for $t\in[0.9,1]$, $i =1,2,3$. The horizontal and vertical axes represent the real and imaginary axes. We see eight sets of three paths (colored red, blue, magenta to represent $i=1,2,3$), each converging to a point that represents the first coordinate of an eigenvalue. We picked the $\lambda_{i1}$-coordinate arbitrarily and could have done the same plot for any of the $15$ coordinates in $(\Lam_1,\Lam_2,\Lam_3,\xv_1,\xv_2,\xv_3) \in \blue{\mathbb{C}^{9} \times \mathbb{P}^1\times \mathbb{P}^1\times \mathbb{P}^1}$. What we are witnessing is a one-dimensional projection of the homotopy path in  \blue{$\mathbb{C}^{9} \times \mathbb{P}^1\times \mathbb{P}^1\times \mathbb{P}^1$} converging to the eight eigenpairs of the MEP. Note that ``one dimension'' here means ``one complex dimension'' which translates to the two real dimensions we see in Figure~\ref{eightEndPaths}.

The left plot shows only the behavior of the \blue{homotopy} path towards the end, i.e., only for $t\in[0.95,1]$. The right plot in Figure~\ref{onePath} shows the full homotopy path, i.e., for all $t\in[0,1]$, of the $\lambda_{i1}$-coordinates for one of the eight solutions of \blue{the left plot}.
\begin{figure}[htb]
\centering
\caption{Left: The $\lambda_{i1}$ coordinates of all eight paths for $t\in[0.\blue{95},1]$, $i=1,2,3$.
Right: The $\lambda_{i1}$ coordinates of \blue{the path highlighted in the left panel} for $t\in [0,1]$, $i=1,2,3$.
The horizontal and vertical axes represent the real and imaginary axes.
}
\includegraphics[trim={0ex, 0ex, 0ex, 0ex},clip,width=0.49\textwidth]{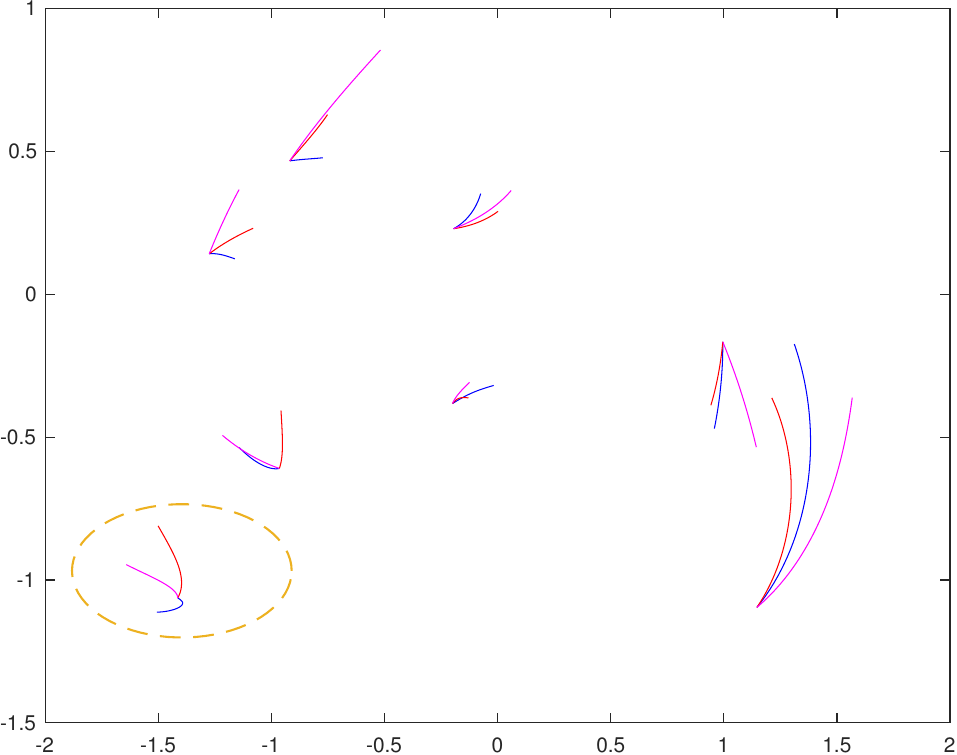}
\includegraphics[trim={0ex, 0ex, 0ex, 0ex}, clip,width=0.49\textwidth]{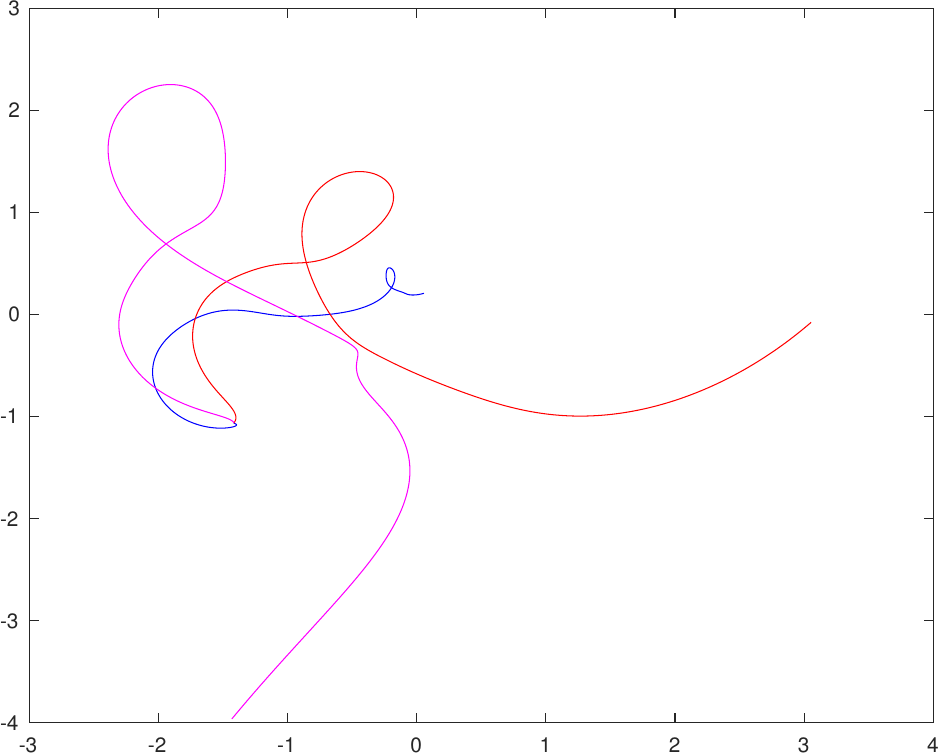}
\label{eightEndPaths}\label{onePath}
\end{figure}
\end{example}

We would like to emphasize that in Example~\ref{eg:3222}, the homotopy path is confined to a six-dimensional subspace of \blue{$\mathbb{C}^{9} \times \mathbb{P}^1\times \mathbb{P}^1\times \mathbb{P}^1$} as only $6 = k(k-1)$ equations involve the path parameter $t$.
The following difference between the fiber product homotopy and diagonal coefficient homotopy for MEP cannot be overstated.
For the former, at most $k(k-1)$  equations involve the path parameter, all of which are \emph{linear}
in $\Lam_1,\dots,\Lam_k$.
For the latter, \blue{$n_1+n_2+\cdots+n_k$} equations involve the path parameter, all of which are \emph{multilinear}.
\blue{Even when $k(k-1)>n_1+\cdots +n_k$, the paths of the fiber product homotopy would still be easier to track than those of the diagonal coefficient homotopy, as we will see later.}

\section{Continuation procedure}\label{sec:continue}

Observe that there are $k$ more variables than equations in \eqref{eq:homotopyPQ} because the eigenvectors are unique only up to scaling, i.e., they are indeed points in projective spaces.  To account for this arbitrary scaling, we fix a generic \emph{affine chart} in each projective space $\mathbb{P}^{n_i-1}$, i.e., by introducing the affine constraints  $\boldsymbol{d}_i^\tp \xv_i=1$ with $\boldsymbol{d}_i\in\mathbb{C}^{n_i}$,  $i=1,\dots,k$.

Let $A_{ij} \in \mathbb{C}^{n_i \times n_i}$, $j=0,1,\dots,k$,
$i=1,\dots,k$ be the matrices of our MEP as in Definition~\ref{definition:MEP}.
Recall the maps   $G_i\colon \mathbb{C}^{k^2}\to\mathbb{C}^{k-1}$
in \eqref{eq:defineG}
and
$L_i\colon\mathbb{C}^{k}\to\mathbb{C}^{k-1}$
in \eqref{eq:solveProduct}.
For $i=1,\dots,k$, we let
$\overline{L}_i \colon \mathbb{C}^{k^2}\to\mathbb{C}^{k-1}$ be defined by
\[
\overline{L}_i(\Lam_1,\dots,\Lam_k) \coloneqq L_i(\Lam_i),
\]
i.e.,  $\overline{L}_i$ extends the domain of $L_i$ from $\mathbb{C}^{k}$ to $\mathbb{C}^{k^2}$.

In the following, we write $\mathbbold{O}_{m\times n} \in \mathbb{C}^{m \times n}$ and $\mathbbold{O}_n \in \mathbb{C}^n$ for the zero matrix and zero vector, and $\mathbbold{1}_n \in \mathbb{C}^n$ for the  all ones vector. To track the homotopy, we use an \emph{Euler--Newton predictor-corrector} method.
The Euler step gives an approximate eigenpair whereas the Newton step refines the approximation:
\begin{description}
\item[\textsc{Euler step:}] Solve $( k^2+n_1+\dots+n_k) \times ( k^2+n_1 + \dots + n_k )$ linear system
\[
\begin{medsize}
\begin{bmatrix}
\diag\bigl(B_{1}(\xv_1),\dots,B_{k}(\xv_k)\bigr)
 &
\diag\bigl(H_{1}(\Lam_{1}),\dots,H_{k}(\Lam_{k})\bigr) \\[1.5ex]
 \mathbbold{O}_{k\times k^{2}}
 &
 \diag(\boldsymbol{d}_{1}^\tp,\dots,\boldsymbol{d}_{k}^\tp) \\[1.5ex]
(1-t)\Jac \overline{L}_1 +t\Jac G_1 &
\mathbbold{O}_{\blue{(k-1)\times (n_{1}+\cdots +n_k)}} \\
\vdots & \vdots \\
(1-t)\Jac \overline{L}_k+t\Jac G_k &
\mathbbold{O}_{\blue{(k-1)\times (n_1+\cdots +n_{k})}}
\end{bmatrix}
 \begin{bmatrix}
\widehat{\Lam}_1\\[-0.5ex]
\vdots\\
\widehat{\Lam}_k\\
\widehat{\xv}_1\\[-0.5ex]
\vdots\\
\widehat{\xv}_k
\end{bmatrix}
=
 \begin{bmatrix}
\mathbbold{O}_{n_1+\dots+n_k+k}\\[1.5ex]
 \blue{\overline{L}_1- G_1}\\
\vdots\\
 \blue{\overline{L}_k -G_k}
\end{bmatrix},
\end{medsize}
\]
where
$B_i(\xv_i) \coloneqq - [A_{i1}\xv_i,\dots,A_{ik}\xv_i] \in \mathbb{C}^{n_i \times k}$, $i =1,\dots,k$.
Here we regard the total derivatives as Jacobian matrices, i.e., $\Jac \overline{L}_i,  \Jac G_i  \in  \mathbb{C}^{(k-1) \times k^2}$, $i=1,\dots,k$.
As $\overline{L}_i$'s and $G_i$'s are affine linear maps, their Jacobians are constant matrices that do not depend on $\Lam_i$'s.

\item[\textsc{Predictor:}] This is given by
\[
\begin{medsize}
\boldsymbol{p} \coloneqq
\begin{bmatrix}
\widetilde{\Lam}_1\\
\vdots\\
\widetilde{\Lam}_k\\
\widetilde{\xv}_1\\
\vdots\\
\widetilde{\xv}_k
\end{bmatrix}
=
\begin{bmatrix}
\Lam_1\\
\vdots\\[0.5ex]
\Lam_k\\
\xv_1\\
\vdots\\
\xv_k
\end{bmatrix}
+
h
\begin{bmatrix}
\widehat{\Lam}_1\\
\vdots\\
\widehat{\Lam}_k\\
\widehat{\xv}_1\\
\vdots\\
\widehat{\xv}_k
\end{bmatrix}  \in \mathbb{C}^{ n_1 + \dots + n_k +k^2}
\end{medsize}
\]
where $h$ is the step size. The predictor is the input to the Newton step.
\item[\textsc{Newton step:}] Solve $(k^2+n_1+\dots+n_k) \times ( k^2 + n_1 + \dots + n_k  )$ linear system
\[
\begin{medsize}
\begin{bmatrix}
 \diag\bigl(B_{1}(\xv_1^{(\nl)}),\dots,B_{k}(\xv_k^{(\nl)})\bigr)
 &
 \diag(H_{1}\bigl(\Lam_{1}^{(\nl)}),\dots,H_{k}(\Lam_{k}^{(\nl)})\bigr) \\[1ex]
 \mathbbold{O}_{k\times k^{2}}&
 \diag(\boldsymbol{d}_{1}^\tp,\dots,\boldsymbol{d}_{k}^\tp) \\[1ex]
(1-t)\Jac \overline{L}_1 +t\Jac G_1 &
\mathbbold{O}_{\blue{(k-1)\times (n_{1}+\cdots +n_{k})}} \\
\vdots & \vdots \\
(1-t)\Jac \overline{L}_k+t\Jac G_k &
\mathbbold{O}_{\blue{(k-1)\times (n_{1}+\cdots +n_{k})}}
\end{bmatrix}
 \begin{bmatrix}
\Delta\Lam_1\\
\vdots\\
\Delta\Lam_k\\
\Delta\xv_1\\
\vdots\\
\Delta\xv_k
\end{bmatrix}
=-\begin{bmatrix}
H_1(\Lam_1^{(\nl)})\blue{ \xv_1^{(\nl)}}\\
\vdots\\
H_k(\Lam_k^{(\nl)})\blue{ \xv_k^{(\nl)}}\\
\boldsymbol{d}_1^\tp \xv_1^{(\nl)}-1\\
\vdots\\
\boldsymbol{d}_k^\tp \xv_k^{(\nl)}-1\\
\boldsymbol{v}^{(\nl)}
\end{bmatrix}
\end{medsize}
\]
where $B_i$ is as defined in the Euler step and
\[
\begin{medsize}
\boldsymbol{v}^{(\nl)} \coloneqq
\begin{bmatrix}
(1-t)\Jac \overline{L}_1 +t\Jac G_1\\
\vdots\\
(1-t)\Jac \overline{L}_k+t\Jac G_k
\end{bmatrix}
\begin{bmatrix}
\Lam_1^{(\nl)}\\
\vdots\\
\Lam_k^{(\nl)}
\end{bmatrix}
+
(1-t)
\begin{bmatrix}
  \mathbbold{1}_{k-1}\\
\vdots\\
 \mathbbold{1}_{k-1}
\end{bmatrix} \in \mathbb{C}^{k(k-1)}.
\end{medsize}
\]
The initial approximation $\bigl[\Lam_1^{(0)},\dots,\Lam_k^{(0)},\xv_1^{(0)},\dots,\xv_k^{(0)}\bigr]^\tp$ is given by $\boldsymbol{p}$.
\item[\textsc{Corrector:}]
This is given by the solution to the Newton step
\begin{equation}\label{eq:corr}
\begin{medsize}
\boldsymbol{c} =
\begin{bmatrix}
\Delta \Lam_1\\
\vdots\\
\Delta\Lam_k\\
\Delta \xv_1\\
\vdots\\
\Delta \xv_k
\end{bmatrix} \in \mathbb{C}^{ n_1 + \dots + n_k + k^2},
\end{medsize}
\end{equation}
which is then added to refine the approximation,
\[
\begin{medsize}
\begin{bmatrix}
 \Lam_1^{(\nl+1)}\\
\vdots\\
\Lam_k^{(\nl+1)}\\
 \xv_1^{(\nl+1)}\\
\vdots\\
 \xv_k^{(\nl+1)}
\end{bmatrix}
=
\begin{bmatrix}
 \Lam_1^{(\nl)}+\Delta \Lam_1\\
\vdots\\
\Lam_k^{(\nl)}+\Delta \Lam_k\\
 \xv_1^{(\nl)}+\Delta \xv_1\\
\vdots\\
 \xv_k^{(\nl)}+\Delta \xv_k
\end{bmatrix}.
\end{medsize}
\]
\end{description}
Our Euler--Newton predictor-corrector method uses a predictor $\boldsymbol{p}$ with size $h$, followed by Newton steps until \blue{the norm of $\boldsymbol{c}$} is sufficiently small or when the maximum number of iterations is reached.
If $t+h<1$, we return to the Euler step.
 If $t+h\geq1$, we update $h$ so that $t+h=1$, and do one final Euler step followed by Newton steps.
For a rough idea of the relative costs, tracking one path, i.e., one eigenpair, typically takes a total of $50$ Euler steps and a total of \blue{$400$} Newton steps.
We provide actual implementation details in Section~\ref{sec:implement}.

The matrices in the Euler and Newton steps will in general look like the one depicted in Figure~\ref{fig:matrix} for a four-parameter eigenvalue problem.
The $k$ row vectors $\boldsymbol{d}_{1}^\tp, \dots, \boldsymbol{d}_{k}^\tp$ near the bottom right corner and the $(k^2 - k) \times k$ block representing $\bigl[(1-t)\Jac \overline{L}_i+t\Jac G_i\bigr]_{i=1}^k$ on the bottom left corner of the matrix are dense
--- recall from Lemma~\ref{lemma:startSystem}, Theorem~\ref{thm:chosenCorrectly}, and the first paragraph of this section that we require the entries in these blocks be generic, i.e., there is zero probability that any of these entries is zero. On the other hand,  if the matrices $A_{ij} \in \mathbb{C}^{n_i \times n_i}$, $j=0,1,\dots,k$, $i=1,\dots,k$, defining the MEP are sufficiently sparse, then the lighter shaded blocks in the top right part of the matrix representing $H_i(\Lam_i)$ (Euler step) or $H_i(\Lam_i^{(\nl)})$ (Newton step), being the sum of $k+1$ sparse matrices, will also be sparse. The darker shaded blocks on the top left part of the matrix representing $B_i(\xv_i)$ (Euler step) or $B_i(\xv_i^{(\nl)})$ (Newton step) are expected to be dense, as the column vectors in these blocks are each a product of a sparse matrix and a dense vector.

\begin{figure}[htb]
\centering
\caption{Pictorial representation of a typical matrix in the Euler and Newton steps. Here $k = 4$ and $n_1 = n_2 = n_3 =n_4$. The darker blocks are almost always dense; the lighter blocks are sparse if the input matrices $A_{ij}$'s are sufficiently sparse; the white areas are always zero.}
\label{fig:matrix}
\begin{tikzpicture}[%
    nodeB/.style={minimum width=.20cm, minimum height=2cm,
        draw, fill=#1, outer sep=0pt},
    nodeB/.default={black!30},
    nodeH/.style={minimum width=2cm, minimum height=2cm,
        draw, fill=#1, outer sep=0pt},
    nodeH/.default={black!5},
    nodeD/.style={minimum width=2cm, minimum height=.075cm,
        draw, fill=#1, outer sep=0pt},
    nodeD/.default={black!30},
    nodeLG/.style={minimum width=.20cm, minimum height=.20cm,
        draw, fill=#1, outer sep=0pt},
    nodeLG/.default={black!30},
    B/.style={nodeB},
    H/.style={nodeH},
    D/.style={nodeD},
    LG/.style={nodeLG},
    mymatrix/.style={draw, column sep=-\pgflinewidth, row sep=-\pgflinewidth, matrix of nodes, inner sep=0pt}
    ]

    \matrix[mymatrix] (U)
    {|[B]|&&&&|[H]|\\
     &|[B]|       &&&&|[H]|    \\
     &&|[B]|&&&&|[H]|\\
     &&&|[B]|&&&&|[H]|\\
     &&&&|[D]| \\
     &&&&&|[D]| \\
     &&&&&&|[D]| \\
     &&&&&&&|[D]| \\
           |[LG]|&|[LG]|&|[LG]|&|[LG]|& \\
           |[LG]|&|[LG]|&|[LG]|&|[LG]|& \\
           |[LG]|&|[LG]|&|[LG]|&|[LG]|& \\
           |[LG]|&|[LG]|&|[LG]|&|[LG]|& \\
           };

\draw [<->] (-4.4,4.7) -- (-3.6,4.7) ;
\draw [<->] (-3.6,4.7) -- (4.4,4.7);
\node[label=north:\rotatebox{0}{$k^2$}] at (-3.9,4.5) {};
\node[label=north:\rotatebox{0}{$n_1+\dots +n_k$}] at (0,4.5) {};

\draw [<->] (4.55,-4.55)  -- (4.55,-3.8) ;
\draw [<->] (4.55,-3.8)  -- (4.55,-3.47) ;
\draw [<->] (4.55,-3.47)  --  (4.55,4.55);

\node[label=right:\rotatebox{0}{$\mathrlap{k^2-k}$}] at (4.4,-4.2) {};
\node[label=right:\rotatebox{0}{$\mathrlap{k}$}] at (4.4,-3.65) {};
\node[label=right:\rotatebox{0}{$\mathrlap{n_1+\dots +n_k}$}] at (4.4,0.6) {};

\end{tikzpicture}
\end{figure}
The Euler--Newton  predictor-corrector method will converge  to the solution of the target system if the step sizes are sufficiently small and the approximate start solution is sufficiently close to the actual start solution \cite[Theorem~5.2.1]{AG1990}.
Quantifying the ``sufficiently small'' and ``sufficiently close'' is still an active area of research, but we will see in Sections~\ref{sec:rand}--\ref{sec:sing} ample numerical evidence of stable convergence to true solutions (even for \blue{dimension-deficient} singular MEPs). Indeed, the results in Sections~\ref{sec:rand}--\ref{sec:sing} are from thousands of MEPs --- a single value in the tables/figures represents an aggregate over tens or hundreds of runs ---  and we did not encounter any instance where Euler--Newton failed to converge.

\section{Implementation}\label{sec:implement}

We implemented\footnote{\url{https://github.com/JoseMath/MEP_Homotopy}}  our method in both \textsc{Matlab} and \textsc{Bertini}/\textsc{Macaulay2}, catering respectively to the numerical computing and symbolic computing communities. \blue{All experiments are carried out on single node, an Intel E5-2680v4 2.4 GHz processor with 28 threads and 56 GB of RAM, in the University of Chicago Research Computing Center.} The parameters below can be readily changed by the user.
\begin{description}
\item[\textsc{Inputs:}] The inputs are the coefficients $A_{ij} \in \mathbb{C}^{n_i \times n_i}$, $j=0,1,\dots,k$, $i=1,\dots,k$, of the polynomial matrices $H_1,\dots,H_k$, as in Definition~\ref{definition:MEP}.

\item[\textsc{Start solutions:}] Solutions to the start system \eqref{eq:startSystem} are obtained as follows.
For each $i$, we set the constant terms of $L_i(\Lam_i)$ to be $-1$, and
generate other coefficients from the standard complex Gaussian distribution using \texttt{randn} (\textsc{Matlab}) or  \texttt{random CC} (\textsc{Macaulay2}). In our \textsc{Matlab} implementation, we determine a null vector $\qv_i\in \ker(\nabla L_i)$ using \texttt{null}, a particular solution $\pv_i$ to $L_i(\Lam_i)=0$ using \texttt{mldivide} (i.e., backslash),
and the associated eigenpairs using \texttt{eig}. In our \textsc{Bertini/Macaulay2} implementation, we use  \texttt{bertiniZeroDimSolve} (\textsc{Bertini}) with default settings  to determine the solutions to $H_i(\Lam_i)\xv_i=0$ and $L_i(\Lam_i)=0$.

\item[\textsc{Continuation:}] The bulk of our computations is in running the Euler--Newton predictor-corrector method  in Section~\ref{sec:continue}. In our \textsc{Matlab} implementation, the linear systems in the Euler and Newton steps are solved using  \texttt{mldivide}. \blue{We use two sets of criteria for controlling the step size and termination. In the \texttt{conservative} criterion, the step size $h$ in the predictor $\boldsymbol{p}$ is updated as follows: if the number of Newton iterations $i_{\textsc{nt}} \leq 2$, we double $h$; if $i_{\textsc{nt}} >8$, we halve $h$;  else we keep the same $h$.}
The \blue{Newton iteration continues until the corrector $\boldsymbol{c}$ satisfies $\|\boldsymbol{c} \|_{\infty} <10^{-9}$ or when number of iterations exceeds $\max\{20, k\max(n_1,\dots ,n_k)+5\}$. In the \texttt{fast} criterion, the} step size $h$ in the predictor $\boldsymbol{p}$ is updated as follows: if the number of Newton iterations $i_{\textsc{nt}} \leq 2$, we double $h$; if $i_{\textsc{nt}} >5$, we halve $h$;  else we keep the same $h$. The \blue{Newton iteration continues until the corrector $\boldsymbol{c}$ satisfies $\|\boldsymbol{c} \|_{2} <10^{-8}$ or when the number of iterations exceeds five.
The maximum and minimum step sizes are $10^{-2}$ and $10^{-6}$ respectively for both two criteria.
The \texttt{fast} criterion is the same one used in \cite{Dong2016}, and we will use it when comparing with the results given by the method in \cite{Dong2016}. Otherwise, we will use the \texttt{conservative} criterion for better accuracy.} These choices are heuristical but may be easily fine-tuned.

\textsc{Bertini} is specifically designed with the homotopy method in mind and these tasks are built-in and automated; users need only specify the configurations they wish to change in the input file. In our \textsc{Bertini}/\textsc{Macaulay2} implementation, we \blue{may} simply use \texttt{runBertini} with default configurations.

\item[\textsc{Stopping conditions:}]
In our \textsc{Matlab} implementation,
we estimate the endpoint of the homotopy as follows: when $t + h > 1$, we perform the Euler step with step size $h = 1 - t$ and refine our endpoint with Newton's method until the change in the update is sufficiently small (the default is $\|\boldsymbol{c}\|_{\infty} <10^{-9}$) or  when the maximum number of iterations is reached (default is set to $\max\{20,k\cdot \max(n_1,\dots,n_k)+5\}$).  Again, \textsc{Bertini} is designed for homotopy method and has a variety of built-in options for estimating endpoint. \blue{Among other things, \textsc{Bertini} has implemented various sophisticated ``endgames'' to identify solutions with multiplicities larger than one and return the values of these multiplicities, a feature that is too involved to replicate in our \textsc{Matlab} implementation.} We use the default ``fractional power series endgame.''
\end{description}

In the next three sections, we present our numerical results on MEPs that are (i) randomly generated (Section~\ref{sec:rand}), (ii) from a real-world application (Section~\ref{sec:mathieu}), and (iii) \blue{dimension-deficient and} singular (Section~\ref{sec:sing}).
We compare speed and accuracy of our method with those of Delta and diagonal coefficient homotopy methods.

\section{Numerical results I: Randomly generated MEPs}\label{sec:rand}

Here we randomly generate our inputs $A_{ij} \in \mathbb{C}^{n_i \times n_i}$, $j=0,1,\dots,k$, $i=1,\dots,k$, from the standard complex Gaussian distribution. For convenience, we assume that $n_1 =\dots = n_k = n$ so that there are just two parameters $k$ and $n$ to consider.
We write $N$ for the number of  eigenpairs, given by Proposition~\ref{prop:N}.

We report the maximum time it takes to track one path, denoted by $t_{\textsc{path}}$, and the average number of Newton iterations  during the path tracking, denoted by $\phi(k,n)$.
The value of $t_{\textsc{path}}$ is important as path-tracking is a task with high parallelism and  $t_{\textsc{path}}$ provides a good estimate of the time it takes to run fiber product homotopy method in parallel. Nevertheless, we are only able to report $t_{\textsc{path}}$ for our \textsc{Matlab} implementation as \textsc{Bertini} deals with path-tracking in a more sophisticated and automated manner that offers users no easy way of determining the total track time of one path.

We investigate the stability of our method and the accuracy of our solutions by examining the backward error as defined in \cite[Section~3]{HP2003}.
The normwise backward error of an approximate eigenpair $(\Lam,\xv_1,\dots,\xv_k)$ of an MEP with coefficients $A_{ij} \in \mathbb{C}^{n_i \times n_i}$, $j=0,1,\dots,k$, $i=1,\dots,k$, and polynomial matrices $H_1,\dots,H_k$, as in Definition~\ref{definition:MEP} is given by
\begin{align*}
\eta(\Lam,\xv_1,\dots,\xv_k) &\coloneqq
\min\{ \varepsilon \in \mathbb{R} :
(H_i(\Lam)+\Delta H_i(\Lam))\xv_i=0, \\
&\qquad\qquad\qquad\lVert \Delta A_{ij}\rVert \leq \varepsilon
\lVert A_{ij}\rVert, \;i=1,\dots,k,\; j=0,\dots,k\},\\
\Delta H_i(\Lam)&\coloneqq \Delta A_{i0}-\sum\nolimits_{j=1}^k\lambda_j\Delta A_{ij},\qquad i =1,\dots,k.
\end{align*}
To compute $\eta(\Lam,\xv_1,\dots,\xv_k)$
we take advantage of \cite[Theorem~2]{HP2003}, which says
\[
\eta(\Lam,\xv_1,\dots,\xv_k)=\max_{i=1,\dots,k}\biggl(\frac{\lVert H_i(\Lam)\xv_i  \rVert}{\lVert A_{i0} \rVert+\sum_{j=1}^k\lvert\lambda_j\rvert\lVert A_{ij}\rVert}\biggr).
\]

\subsection{Fixed $k$, varying $n$}

Here we fix $k=3$ and $n_1=n_2=n_3 = n$.
For each value of $n$ we generate ten three-parameter eigenvalue problems.  Note that the expected number  of eigenpairs in these problems is $N = n^3$. \blue{As our \textsc{Matlab} implementation of fiber product homotopy method computes in parallel with $28$ threads, the wall time reported in Table~\ref{table:Random system} is estimated by multiplying the real wall time by $28$.}
In Table~\ref{table:Random system},
we see that our method is faster than the timings reported for the diagonal coefficient homotopy method in \cite{Dong2016}. The Delta method fails for larger values of $n$ --- \blue{it} crashes with \blue{an} out-of-memory error in every instance when $n=\blue{30}$.
\begin{table}[hbt]
\centering
\caption{Elapsed timings (in seconds) for Delta method, fiber product homotopy method (averaged over ten runs), and diagonal coefficient homotopy method \cite{Dong2016}.}
\label{table:Random system}
\setlength{\extrarowheight}{2pt}
\begin{tabular}{c|c|c|c|c|c|c}
$n$ & $N$ &\multicolumn{3}{c|}{Wall time} & $t_{\textsc{path}}$& $\phi(3,n)$\tabularnewline
\hline
 & $n^{3}$ &  Delta Mtd.& Fiber Prod.&Diag. Coeff.& \multicolumn{2}{c}{Fiber Prod.} \tabularnewline
\hline
\hline
10 & 1000 & \blue{1.44} & \blue{264.38} & 779.47& \blue{0.22} & \blue{386} \tabularnewline
\hline
15 & 3375 & \blue{19.14} & \blue{906.00} & 2888.73& \blue{0.25} & \blue{398} \tabularnewline
\hline
20 & 8000 & \blue{150.59} & \blue{2547.29} & 7857.44& \blue{0.31} & \blue{402} \tabularnewline
\hline
25 & 15625 & \blue{815.9125} & \blue{5905.11} & 17169.53 & \blue{0.37} & \blue{408}\tabularnewline
\hline
30 & 27000 & \textsc{failed} & \blue{12820.89} & 32786.64 & \blue{0.46} & \blue{418}
\end{tabular}
\end{table}

Our results for backward errors are presented in Figure~\ref{fig:Random system}, where it is clear that fiber product homotopy method (blue plot) has significantly smaller backward error than the Delta method (red plot) in our numerical experiments.
\blue{The experiments in Table~\ref{table:Random system} and Figure~\ref{fig:Random system} find \emph{all} eigenpairs by tracking all start solutions of  \eqref{eq:startSystem}.}
\begin{figure}[hbt]
\centering
\caption{The \blue{$\log_{10}$} backward \blue{errors} (vertical axis) of fiber product method \blue{(blue)} and the Delta method \blue{(red)} are plotted against the eigenvalues (horizontal axis) ordered by increasing norm from left to right.}
\includegraphics[trim={0ex, 0ex, 0ex, 0ex}, clip,width=0.49\textwidth]{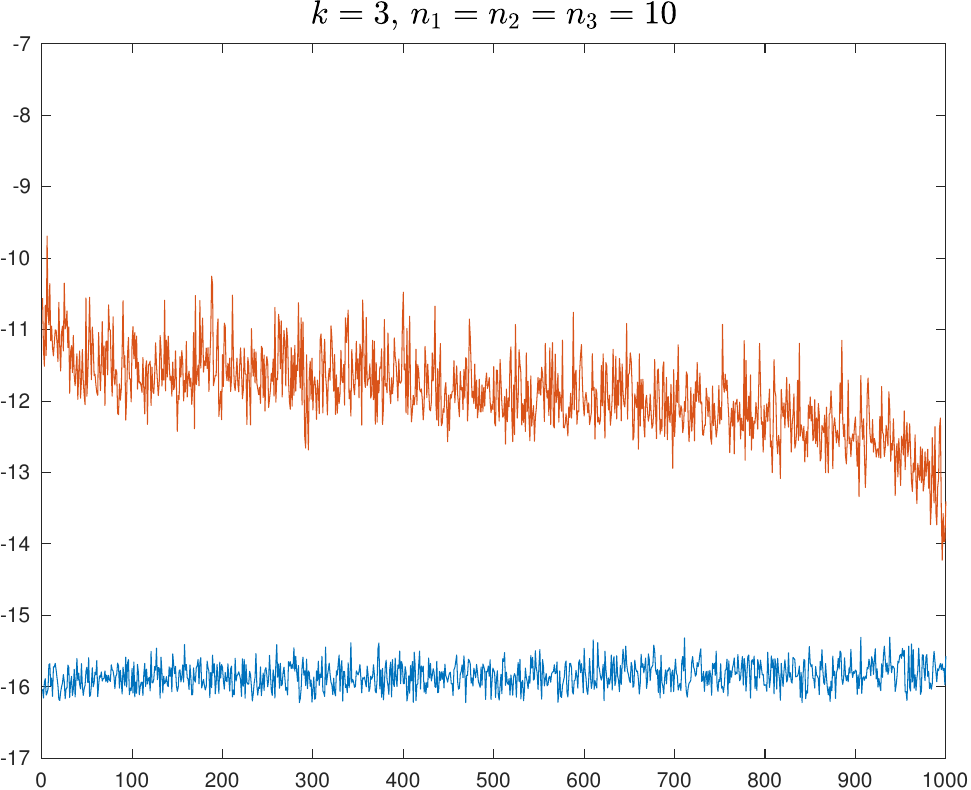}
\includegraphics[trim={0ex, 0ex, 0ex, 0ex}, clip,width=0.49\textwidth]{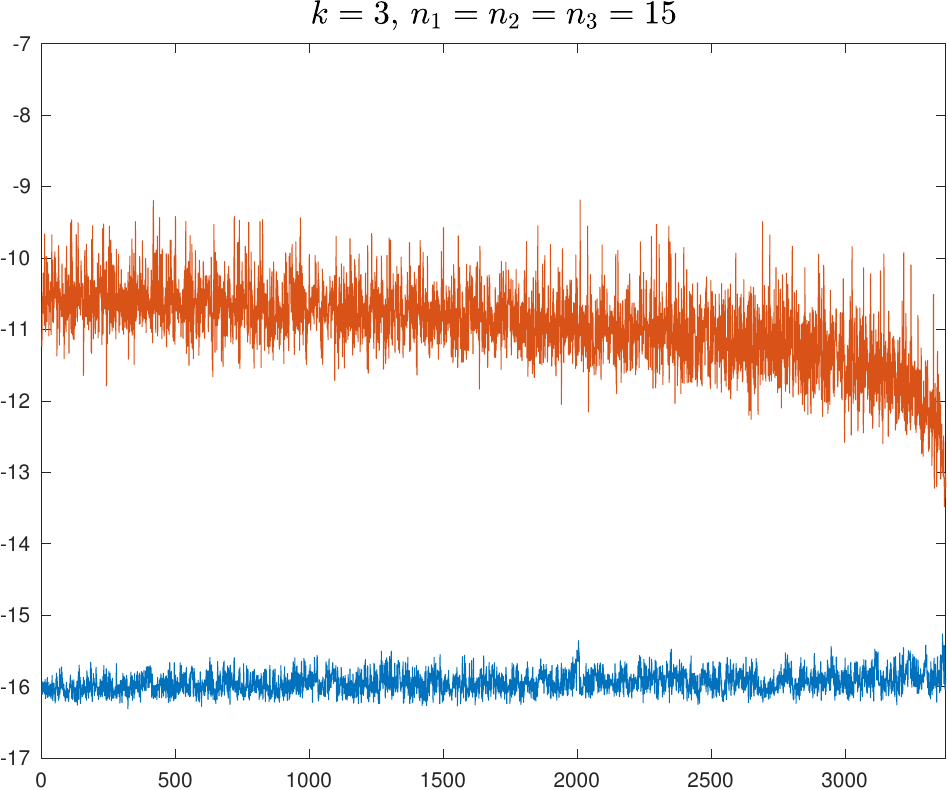}
\includegraphics[trim={0ex, 0ex, 0ex, 0ex}, clip,width=0.49\textwidth]{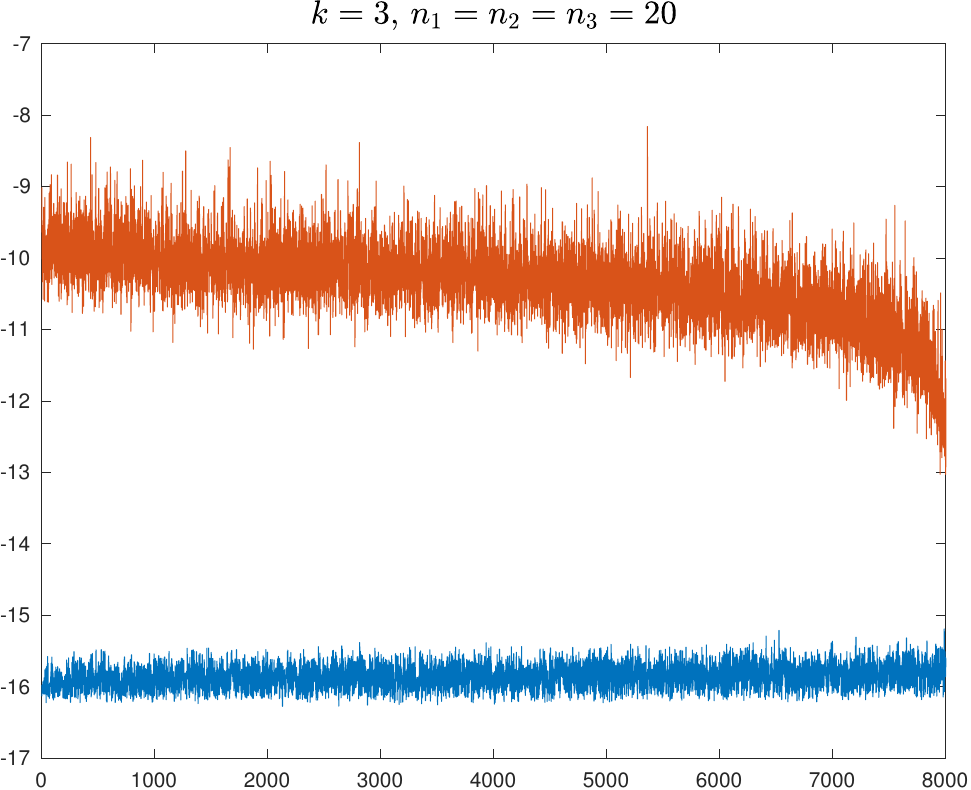}
\includegraphics[trim={0ex, 0ex, 0ex, 0ex}, clip,width=0.49\textwidth]{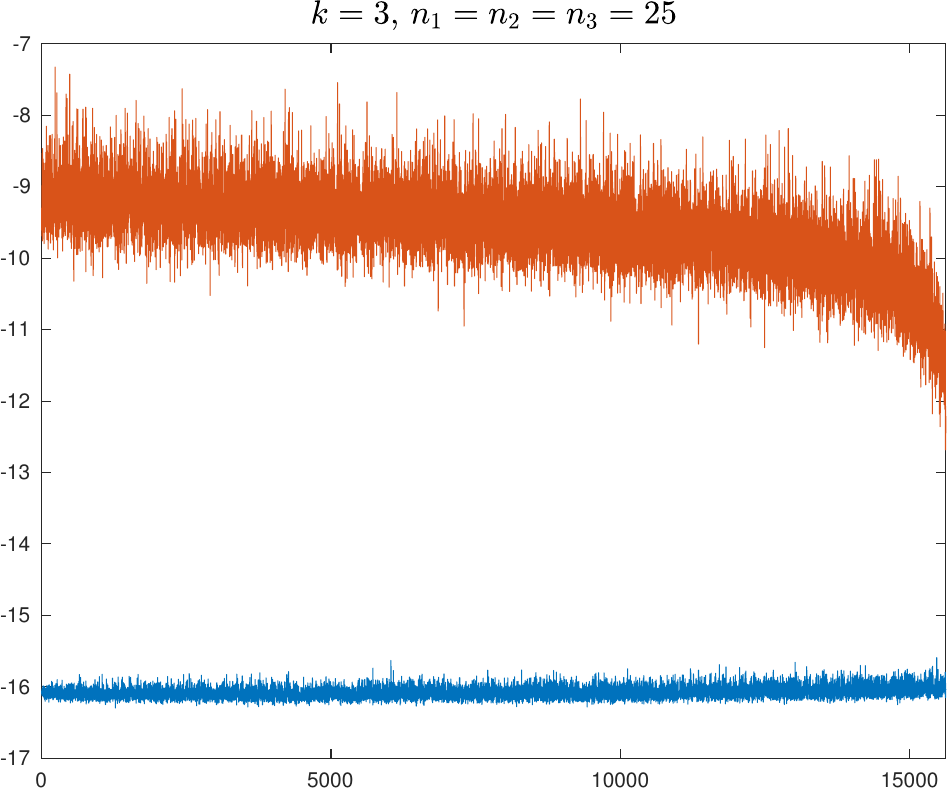}
\label{fig:Random system}
\end{figure}

In our next experiment, our goal is to examine the effect of a substantial increase in $n$ on speed and accuracy of our fiber product homotopy method.
\blue{Finding all eigenpairs would take too long and serves little purpose for this next experiment.}
Instead, we randomly generate $100$ MEPs for each $n$ and track \emph{one} random start solution.
The timings in Table~\ref{table:fixKonePath} are rough estimates of the time it would have taken to find all eigenpairs, obtained by multiplying average wall time for one randomly chosen eigenpair by $n^3$.
The backward errors in Table~\ref{table:fixKonePath} are for  one randomly  chosen eigenpair (and not multiplied by $n^3$). Evidently, increasing $n$ has negligible effect on  backward errors, which are all splendidly small --- on the order of $10^{-15}$ or less.

\begin{table}[h]
\centering
\caption{Timings (in seconds) and accuracy for fiber product homotopy method ($100$ runs). }
\label{table:fixKonePath}
\setlength{\extrarowheight}{2pt}
\begin{tabular}{c|c|c|c|c|c|c}
$n$  &\multicolumn{3}{c|}{Wall time} & \multicolumn{3}{c}{Backward error}\tabularnewline
\hline
 &  Best & Avg & Worst&   Best & Avg & Worst \tabularnewline
\hline
\hline
30 & \blue{0.71} & \blue{1.12} &  \blue{4.45} & \blue{$7.47\times 10^{-17}$} & \blue{$1.81\times 10^{-16}$} & \blue{$ 1.42\times 10^{-15}$} \tabularnewline
\hline
70 & \blue{2.88} & \blue{4.27} & \blue{8.85} & \blue{$8.18\times 10^{-17}$}  & \blue{$1.66\times 10^{-16}$} & \blue{$1.26\times 10^{-15}$} \tabularnewline
\hline
150 & \blue{16.08} &  \blue{29.07} & \blue{91.45} & \blue{$1.05\times 10^{-16}$} & \blue{$1.96\times 10^{-16}$} & \blue{$1.36\times 10^{-15}$}
\end{tabular}
\end{table}

\subsection{Fixed $n$, varying $k$}

This time we fix $n_1 =\dots =n_k =3$ and vary $k$ from $3,\dots,9$. The expected number of eigenpairs is then $N = 3^k$. If the reader is wondering why we do not increase $k$ in a more drastic manner, note that increasing $k$ produces a corresponding exponential increase in the number of solutions ---
for each randomly generated MEP, there are $3^k$ eigenpairs. Also, unlike changing $n$, which just changes the dimension of the problem, changing $k$ gives a \emph{different class} of problems --- for example, a two-parameter eigenvalue problem is qualitatively different from a one-parameter eigenvalue problem (i.e., a GEP) --- and each $k$ deserves a careful examination.
Compared to the numbers for diagonal coefficient homotopy method in \cite[Table~3]{Dong2016}, the numbers in Table~\ref{table:Fix_n_vary_k} show that the fiber product homotopy method is significantly faster and also more stable in the sense that every path converged and did not need to be rerun.

The reader is reminded that fiber product homotopy method tracks $k$ paths each carrying a copy of the eigenvalue; these $k$ copies $\Lam_1,\dots,\Lam_k$ converge to the same eigenvalue $\Lam_1 =\dots = \Lam_k$ if the method runs correctly (see Example~\ref{eg:3222}). This is indeed the case as the reader can see from the near zero values  of \blue{$\delta(k) \coloneqq \max_{2\le i\le k}(\lVert\Lam_1-\Lam_i\rVert_1)$} reported in Table~\ref{table:Fix_n_vary_k}.
\begin{table}[h]
\centering
\caption{Elapsed timings (in seconds) for Delta method and the two homotopy methods, all  in \textsc{Matlab}. Accuracy for our method. Here $\delta(k) = \max_{2\le i\le k}(\lVert\Lam_1-\Lam_i\rVert_1)$.
}
{\footnotesize
\label{table:Fix_n_vary_k}
\setlength{\extrarowheight}{2pt}
\begin{tabular}{c|c|c|c|c|c|c|c}
$k$ & $N$ & \multicolumn{3}{c|}{Wall time} & \blue{$\delta(k)$} & $\phi(k,\blue{3})$&$t_{\textsc{path}}$ \tabularnewline
\hline
 & $3^k$ & Delta Mtd.& Fiber Prod.& Diag. Coeff. &  \multicolumn{3}{c}{Fiber Prod.}\tabularnewline
\hline
\hline
4 & 81 & \blue{0.13} & \blue{50.81} &77.63& \blue{$2.90\times 10^{-15}$} & \blue{382} & \blue{0.27}\tabularnewline
\hline
5 & 243 & \blue{0.19} & \blue{89.44} & 288.24 & \blue{$5.43\times 10^{-15}$} & \blue{403} & \blue{0.26}\tabularnewline
\hline
6 & 729 & \blue{1.32} & \blue{269.92} & 1089.10 & \blue{$1.37\times 10^{-14}$} & \blue{433} & \blue{0.32} \tabularnewline
\hline
7 & 2187 & \blue{12.30} & \blue{920.54} &3979.94& \blue{$ 3.18\times 10^{-14}$} & \blue{443} & \blue{0.40}\tabularnewline
\hline
8 & 6561 & \blue{172.44} & \blue{3355.81} &13445.96& \blue{$7.35\times 10^{-13}$} & \blue{457} & \blue{0.49} \tabularnewline
\hline
9 & 19683 & \blue{\textsc{failed}} & \blue{12893.37} & 48624.78& \blue{$1.06\times 10^{-13}$}  & \blue{466} & \blue{0.64}
\end{tabular}}
\end{table}

\subsection{Comparisons on Bertini}

The results in Tables~\ref{table:Random system} and \ref{table:Fix_n_vary_k} compare the \textsc{Matlab}  implementation of fiber product homotopy method  \eqref{eq:homotopyPQ}  to the results for diagonal coefficient homotopy method  \eqref{eq:diagPQ} reported in \cite{Dong2016}. Here we will compare them on \textsc{Bertini}.\footnote{We have also used a \textsc{Macaulay2} \cite{M2} package \cite{bertini4M2} to produce the relevant input files.}  As \textsc{Bertini} is designed for homotopy continuation methods, we do not include a non-homotopy method like Delta method in our comparison on this platform.

In this numerical experiment, we look at a two-parameter eigenvalue problem, i.e., $k=2$, with $n_1=n_2 =n$. The dimension of the problem is thus $N =n^2$. From Table~\ref{table:BertiniTest}, the fiber product homotopy method is consistently faster that the diagonal coefficient homotopy method in \textsc{Bertini}. We include the timings of our \textsc{Matlab} implementation of the fiber product homotopy method in the last column for comparison.
\begin{table}[h]
\centering
\caption{Elapsed timings (in seconds) for comparing homotopy methods and implementations.
}
\label{table:BertiniTest}
\setlength{\extrarowheight}{2pt}
\begin{tabular}{c|c|c|c|c}
$n$ & $N$  &  \multicolumn{2}{c|}{\textsc{Bertini}} & \textsc{Matlab}\tabularnewline
\hline
 & $n^2$ & Diag. Coeff. & Fiber Prod. & Fiber Prod.\tabularnewline
\hline
\hline
$5$ & 25 & 0.25 & 0.45 & 0.61 \tabularnewline
\hline
$15$ & 225 & 31.93 & 29.36 &  3.02  \tabularnewline
\hline
$30$ & 900   & 1158.61 & 882.68 & 18.99\tabularnewline
\hline
$50$ & 2500 & 18661.62 & 14840.28  & 110.79
\end{tabular}
\end{table}

\section{Numerical results II: Mathieu's systems}\label{sec:mathieu}

An example where multiparameter eigenvalue problems surface is Mathieu's systems, which in turn arises from studies of  vibration of a fixed elliptic membrane \cite{PLESTENJAK2015585,Volkmer1988}.
We will test our \textsc{Matlab} implementation of fiber product homotopy method against the Delta method on this problem.

We refer the reader to \cite[Section~2]{PLESTENJAK2015585} for a discussion of how a coupled system of two-point boundary value problems, representing Mathieu's angular and radial equations, yields a two-parameter eigenvalue problem (thus $k =2$) upon Chebyshev collocation discretization. The dimensions of the matrices $n_1$ and $n_2$ correspond to the number of points used in the discretization.

In Figure~\ref{fig:Application_1_Mathieu}, we present  accuracy result for a two-parameter eigenvalue problem with $n_1 = 18$ and $n_2 =38$ coming from a Mathieu system. The horizontal axis represents the $n_1n_2 = 684$ eigenvalues, ordered from the smallest to largest by the norm of $\Lam_1$.
\begin{figure}[htb]
\centering
\caption{The \blue{$\log_{10}$} backward \blue{errors} (vertical axis) of fiber product method \blue{(blue)} and  Delta method \blue{(red)} plotted against eigenvalues (horizontal axis) ordered from smallest to largest.
Left figure: $81$ Newton iterations, taking \blue{$19$ seconds}. Right figure: \blue{$25$} Newton iterations \blue{at the end of the path}, taking \blue{$16$ seconds}.
}
\includegraphics[trim={0ex, 0ex, 0ex, 0ex}, clip,width=0.49\textwidth]{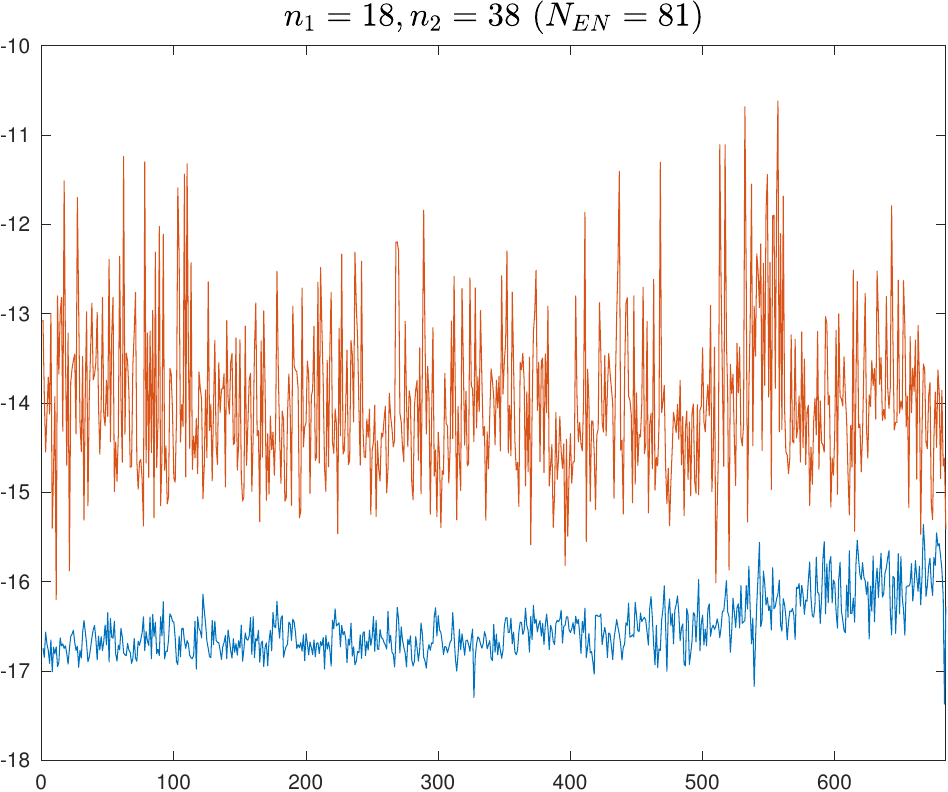}
\includegraphics[trim={0ex, 0ex, 0ex, 0ex}, clip,width=0.49\textwidth]{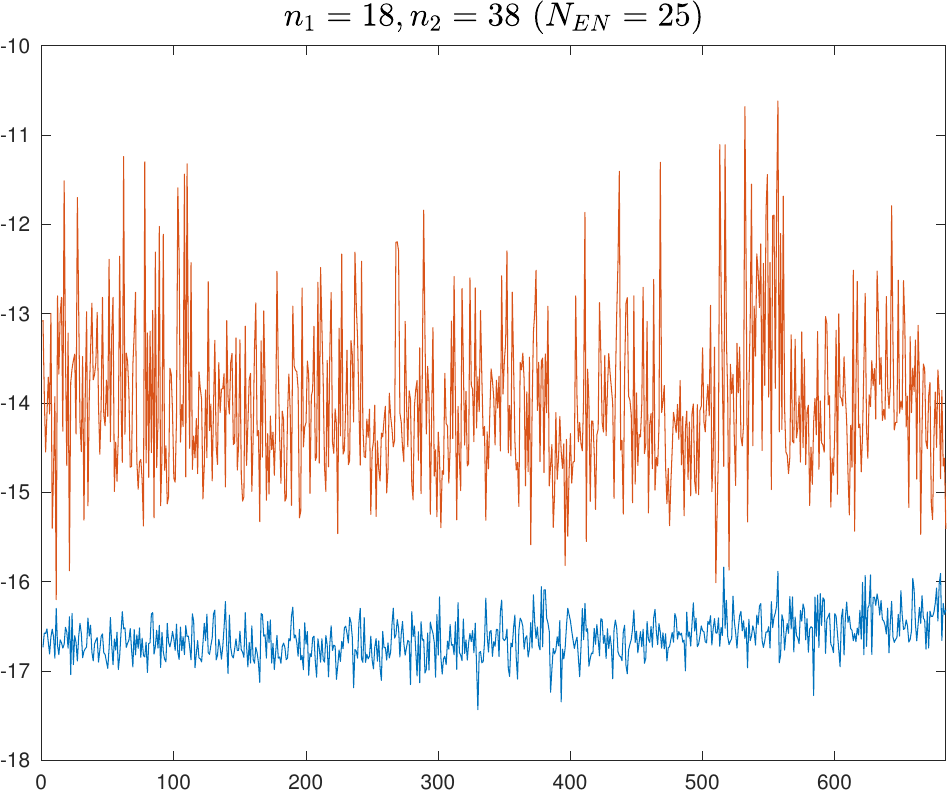}
\label{fig:Application_1_Mathieu}
\end{figure}

The vertical axis measures backward errors on a log scale. The blue plot shows the backward error of our fiber product homotopy method whereas the red plot is that for the Delta method. With few exceptions, fiber product homotopy produces significantly more accurate results than the Delta method by orders of magnitude.  Indeed, every eigenpair computed with our method here has a backward error that is less than $10^{-15}$.

The left and right figures in Figure~\ref{fig:Application_1_Mathieu} differ by \blue{$N_{\textsc{en}}$,} the number of Newton iterations used in the Newton step \blue{at the end of the path} --- the left plot uses the \blue{conservative} stopping condition of
$81=\max(20,k\max\{n_1,\dots,n_k\}+5)$ iterations whereas the right plot uses an early stopping condition of \blue{$25$} iterations. We see no discernible difference in the backward errors but the left plot took \blue{19 seconds with 28 threads} whereas the right plot took only \blue{16 seconds} to compute. This suggests that there is room for further fine-tuning to improve speed without sacrificing stability.

We next employ Shub--Smale $\alpha$-theory \cite{BCSS1998} to certify the quadratic convergence of the Newton steps in a neighborhood of the end point. Briefly, in this theory there is a function $\alpha$ that takes a polynomial system $f \colon \mathbb{C}^{n}\to\mathbb{C}^n$ and an approximate solution $z\in\mathbb{C}^n$ as its input and returns a positive real number.
If $\alpha(f,z)$ is less than the constant  $(13-3\sqrt{17})/4$, then one has certified quadratic convergence of Newton's method on the approximate solution $z$ to the polynomial system $f$ \cite[Theorem~2, p.~160]{BCSS1998}.
With our default tolerances we  are able to certify the smallest $550$ of the $684$ eigenvalues --- see Figure~\ref{fig:Certification_N20}.
\begin{figure}[htb]
\centering
\caption{
Computed bound (vertical axis) on $\log_{\blue{10}} \alpha$ for approximate eigenpairs of Mathieu MEP plotted against eigenvalues (horizontal axis) ordered from smallest to largest.
 If the computed bound is below $\log_{\blue{10}} [(13-3\sqrt{17})/4]$ (red dashed line), the eigenpair is certified.
Left figure: $81$ Newton iterations. Right figure: $25$ Newton iterations.
}
\includegraphics[trim={0ex, 0ex, 0ex, 0ex}, clip,width=0.49\textwidth]{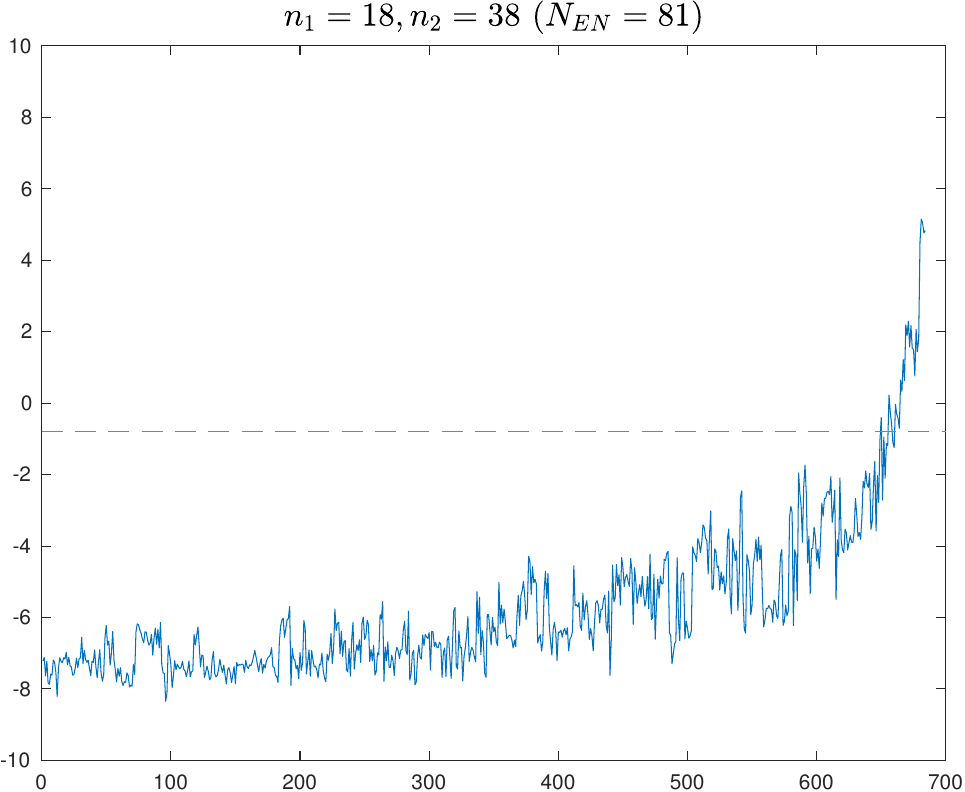}
\includegraphics[trim={0ex, 0ex, 0ex, 0ex}, clip,width=0.49\textwidth]{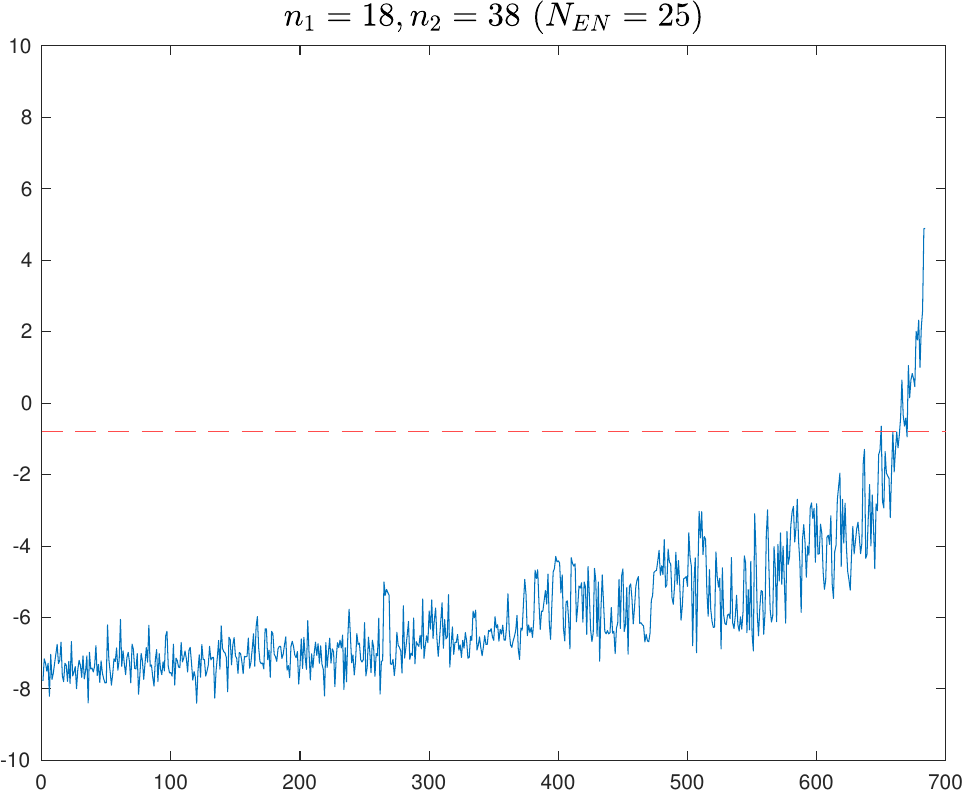}
\label{fig:Certification_N20}
\end{figure}

\section{Numerical results III: \blue{Dimension-deficient} singular MEPs}\label{sec:sing}

One of the thorniest issues in solving MEPs is singularity --- a \emph{singular} MEP  \cite[Section~2]{Hochstenbach2012} is one with \blue{singular Delta matrix}. They have fewer than the expected $n_1\cdots n_k$ eigenpairs; in particular, they are not regular.
That a singular $k$-parameter eigenvalue problem presents computational difficulties is already evident when $k = 1$: a GEP with a singular matrix pencil $A - \lambda B$ is well-known to be challenging computationally. \blue{Among the singular MEPs, we shall impose a further challenge of \emph{dimension-deficiency}, i.e., where the intrinsic dimensions $(d_1,\dots,d_k)<(n_1,\dots,n_k)$.}

\blue{Such dimension-deficient singular} MEPs are inevitable when one linearizes \emph{quadratic multiparameter eigenvalue problems (QMEPs)} \cite{Hochstenbach2012}: Given quadratic polynomial matrices
\begin{equation}\label{eq:qmepQ}
\begin{aligned}
Q_{1}(\lambda,\mu) & \coloneqq   B_{00}+  \lambda B_{10}+\mu B_{01}+\lambda^{2}B_{20}+\lambda\mu B_{11}+\mu^{2}B_{02},\\
Q_{2}(\lambda,\mu) & \coloneqq   C_{00}+  \lambda C_{10}+\mu C_{01}+\lambda^{2}C_{20}+\lambda\mu C_{11}+\mu^{2}C_{02},
\end{aligned}
\end{equation}
where $B_{ij} \in\mathbb{C}^{n_1 \times n_1}$ and $C_{ij} \in \mathbb{C}^{n_2 \times n_2}$, $i, j=0,1,2$,
solve
\[
Q_1(\lambda,\mu)\xv_1=0,\qquad Q_2(\lambda,\mu)\xv_2=0
\]
for all possible $\lambda,\mu \in \mathbb{C}$ and nonzero $\xv_1 \in \mathbb{C}^{n_1}$, $\xv_2 \in \mathbb{C}^{n_2}$. It is straightforward to generalize this to a quadratic $k$-parameter eigenvalue problem but we will only study the case $k = 2$ here.

The quadratic two-parameter eigenvalue problem is mathematically equivalent to a  two-parameter eigenvalue problem \cite{Hochstenbach2012}:
\begin{equation}\label{eq:qmepH}
\begin{aligned}
H_{1}(\lambda,\mu) & =  \begin{bmatrix}
B_{00} & B_{10} & B_{01}\\
0 & -I & 0\\
0 & 0 & -I
\end{bmatrix}+\lambda\begin{bmatrix}
0 & B_{20} & B_{11}\\
I & 0 & 0\\
0 & 0 & 0
\end{bmatrix}+\mu\begin{bmatrix}
0 & 0 & B_{02}\\
0 & 0 & 0\\
I & 0 & 0
\end{bmatrix},\\
H_{2}(\lambda,\mu) & =  \begin{bmatrix}
C_{00} & C_{10} & C_{01}\\
0 & -I & 0\\
0 & 0 & -I
\end{bmatrix}+\lambda\begin{bmatrix}
0 & C_{20} & C_{11}\\
I & 0 & 0\\
0 & 0 & 0
\end{bmatrix}+\mu\begin{bmatrix}
0 & 0 & C_{02}\\
0 & 0 & 0\\
I & 0 & 0
\end{bmatrix}.
\end{aligned}
\end{equation}
Whereas the original coefficient matrices in $Q_i$ are in $\mathbb{C}^{n_i\times n_i}$,
the  coefficient matrices of $H_i$ are in $\mathbb{C}^{3n_i\times 3n_i}$, $i =1,2$. However, the real catch is that the two-parameter eigenvalue problem \eqref{eq:qmepH} is singular \blue{and} its intrinsic dimension $(2n_1,2n_2)$ is strictly smaller than $(3n_1, 3n_2)$.

In our experiment, we set $n_1 = n_2 = n$, randomly generate $B_{ij},C_{ij} \in \mathbb{C}^{n \times n}$, and compare fiber product and diagonal coefficient homotopy methods in \textsc{Bertini}.
The difficulty of \blue{dimension-deficient} MEP is conspicuously reflected in our observed results: The diagonal coefficient homotopy method has $9n^2$ start solutions and $5n^2$ of the paths tracked does not
converge for all values of $n$ we tested, which we recorded in Table~\ref{table:BertiniTime}. On the other hand, our fiber product homotopy method did not produce a single divergent path.
\blue{For $n=40$, the wall time of the diagonal coefficient homotopy method in \textsc{Bertini} is estimated based on the number of paths it tracked when the fiber product homotopy method finishes.}
\begin{table}[h]
\centering
\caption{Number of divergent paths and elapsed timings (in seconds) for the diagonal coefficient (DC), fiber product (FP) homotopy methods, and Delta method ($\Delta$M).  $t_{\textsc{path}}$ is the time taken to track a single path, i.e., roughly how fast FP would have run in parallel.
}
{\footnotesize
\label{table:BertiniTime}
\setlength{\extrarowheight}{2pt}
\begin{tabular}{c|c|c|c|c|c|c|c|c|c}
$n$ & $N$ & \multicolumn{3}{c|}{\# of divergent paths} & \multicolumn{4}{c|}{Wall time} & $t_{\textsc{path}}$ \tabularnewline
\hline
 & $4n^2$ & \multicolumn{2}{c|}{\textsc{Bertini}} & \textsc{Matlab} & \multicolumn{2}{c|}{\textsc{Bertini}} & \multicolumn{3}{c}{\textsc{Matlab}}
\tabularnewline
\hline
 &  & DC &  FP & FP &  DC &  FP & $\Delta$M & \multicolumn{2}{c}{FP} \tabularnewline
\hline
\hline
2 & 16 & 20 & 0 &0& \blue{0.49} & \blue{0.38} &  \blue{0.07} & \blue{0.55} & \blue{0.19} \tabularnewline
\hline
5 & 100 & 125 & 0 &0& \blue{23.72} & \blue{8.60} & 0.11 & \blue{1.36} & \blue{0.22}\tabularnewline
\hline
10 & 400 & 500 & 0 &0& \blue{801.24} & \blue{244.01} & \blue{0.95} & \blue{5.41} & \blue{0.30}\tabularnewline
\hline
20 &  1600& 2000 & 0 &0& \blue{38761.58} & \blue{14079.69} & \blue{24.98} & \blue{56.56} & \blue{0.79} \tabularnewline
\hline
40 & 6400 & 8000 & 0 &0& \blue{\textsc{33 days}} & \blue{\textsc{8 days}}  & \blue{1156.38} & \blue{1133.90} & \blue{4.40}
\end{tabular}}
\end{table}

We next compare fiber product homotopy and Delta methods \cite{MP2010} in \textsc{Matlab}. From the backward errors in Figure~\ref{fig:Singular system} for \blue{dimension-deficient} singular MEPs, we see a similar pattern as in the nonsingular MEPs in Figure~\ref{fig:Random system} --- fiber product homotopy method is more accurate than the Delta method by orders of magnitude. More importantly,  Delta method loses accuracy as $n$ increases but fiber product homotopy method maintains it at $10^{-16}$ even after $n$ is doubled four times.

From the last three columns of  Table~\ref{table:BertiniTime}, we see that while the Delta method beats fiber product homotopy method in terms of speed for smaller values of $n$, the situation reverses when $n = 40$.  Furthermore, the value of $t_{\textsc{path}}$ in the last column indicates that had we run  fiber product homotopy method in
parallel on a machine with sufficiently many cores,  all eigenpairs can be obtained in seconds. On the other hand, \blue{even though the program \texttt{eig} in \textsc{Matlab} has the ability to use multiple cores and processors for the Delta method, parallelization is not straightforward as the problem involves $k$ large Delta matrices of sizes $n_1\cdots n_k \times n_1\cdots n_k $ that need to be formed and stored in advance. As such, the Delta method is infeasible for large matrices.}

\begin{figure}[htb]
\centering
\caption{The \blue{$\log_{10}$} backward errors (vertical axis) of fiber product homotopy (blue) and Delta methods (red) for the \blue{dimension-deficient} singular MEP \eqref{eq:qmepH}, plotted against eigenvalues (horizontal axis) ordered in increasing norm from left to right.}
\includegraphics[trim={0ex, 0ex, 0ex, 0ex}, clip,width=0.49\textwidth]{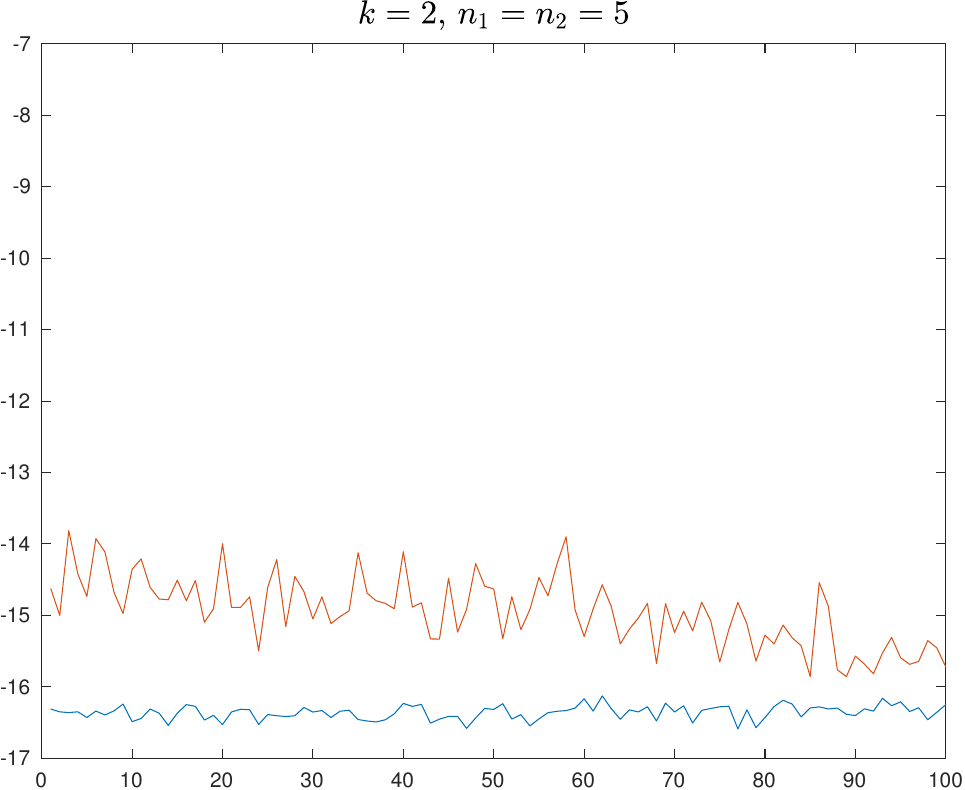}
\includegraphics[trim={0ex, 0ex, 0ex, 0ex}, clip,width=0.49\textwidth]{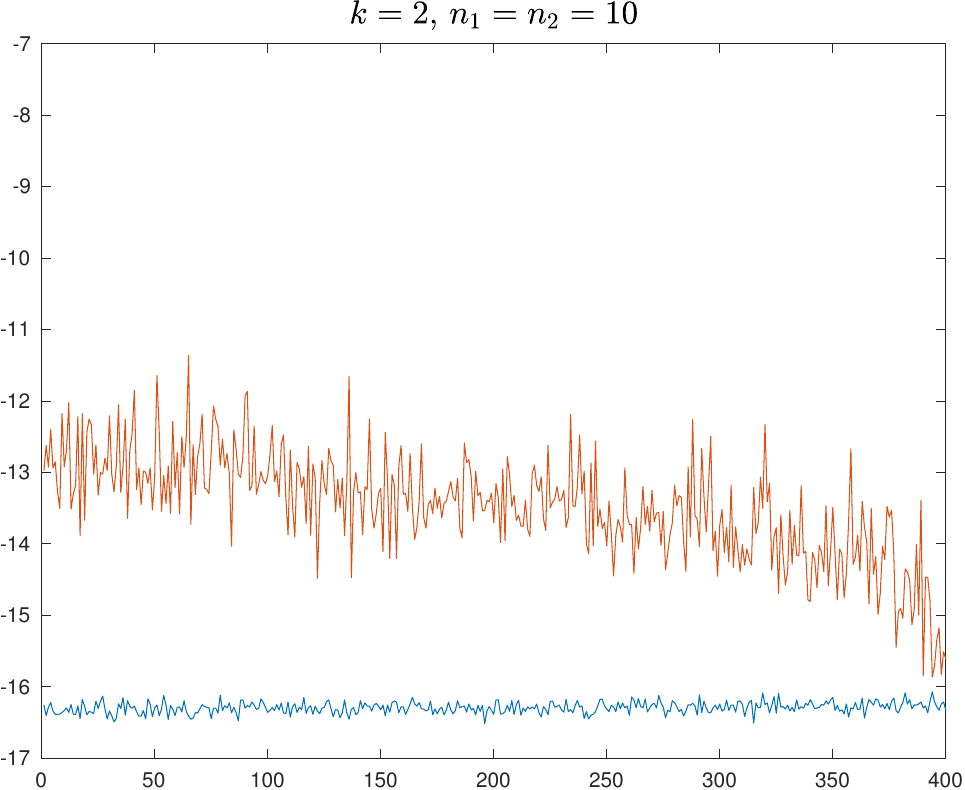}
\includegraphics[trim={0ex, 0ex, 0ex, 0ex}, clip,width=0.49\textwidth]{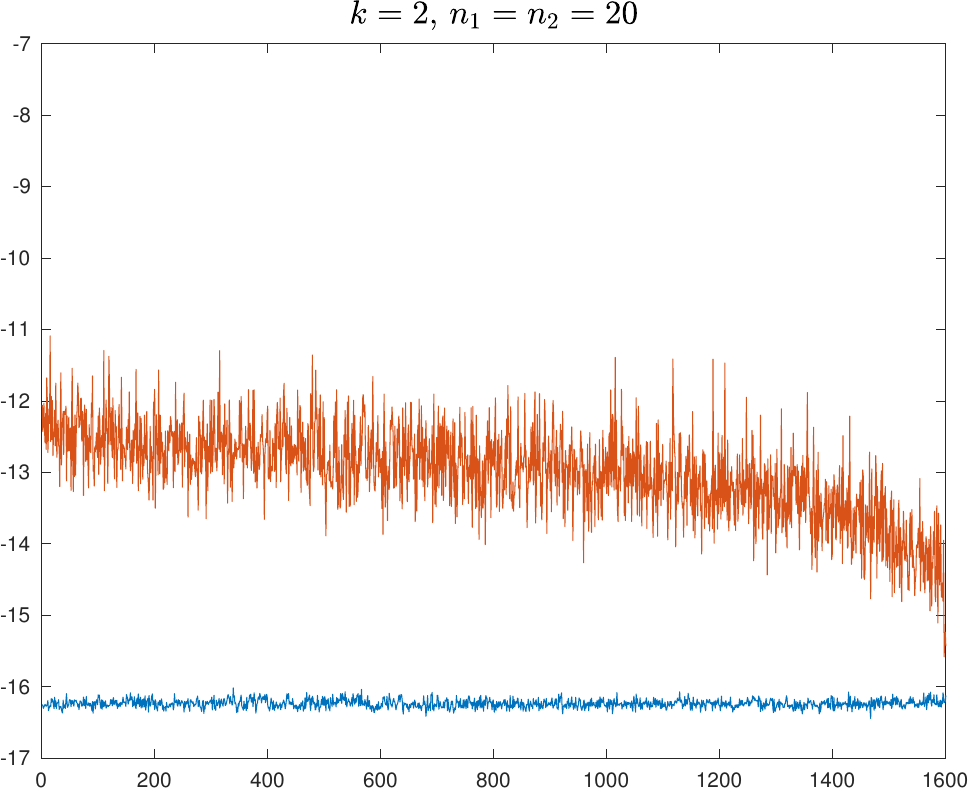}
\includegraphics[trim={0ex, 0ex, 0ex, 0ex}, clip,width=0.49\textwidth]{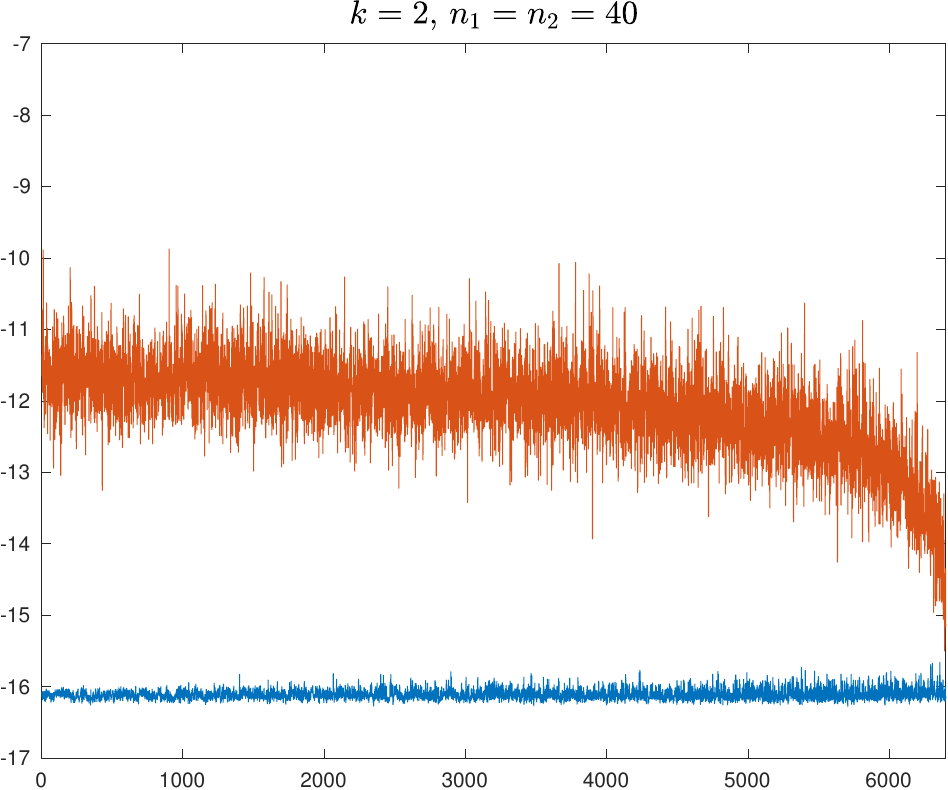}
\label{fig:Singular system}
\end{figure}

\section{Conditioning of the fiber product MEP and standard MEP}\label{sec:condition}

Recall that condition number \cite[Chapter~14]{ConditionBook} depends on the \emph{problem}  --- two different problems that always give the same solution will in general have different condition numbers. The condition number for MEP  as defined in \cite{HP2003} applies to the standard formulation of MEP in \eqref{eq:mep1}. Since we are solving a different problem \eqref{eq:equivMEP}, or more generally \eqref{eq:mep3}, albeit one that gives the same solutions as \eqref{eq:mep1}, it will have a different condition number.

In this section, we discuss the condition number for the fiber product MEP \eqref{eq:equivMEP} and compare it to the condition number for the standard  MEP as defined in \cite{HP2003}. They are expected to be quite different by the common interpretation of condition number as measuring the change in output produced by a change in input. The standard condition number for MEP in \cite{HP2003}  measures the change in eigenvalues (output) produced by a change in the coefficient matrices (input).  In fiber product homotopy method,
we do not deform these coefficient matrices; instead, the relevant notion of condition number is one that measures the change in eigenvalues (output) produced by a change in the linear space defined by $G_1 = \dots = G_k = 0$ as in \eqref{eq:mep3}  (input). So in our setting, we require the condition number of a variety intersected with a varying linear subspace.

In the following, by a \emph{projective linear subspace} of a projective space, we mean
$\{ \pi(\xv) \in \mathbb{P}^n : M\xv = 0, \; \xv \ne 0 \}$  for some $M \in \mathbb{C}^{m \times (n+1)} $ and where $\pi \colon \mathbb{C}^{n+1} \setminus \{0\} \to \mathbb{P}^n$ is the canonical projection. Henceforth,  we will write $\mathbb{M} \coloneqq \pi \bigl(\ker(M)  \setminus \{0\}\bigr)$ for the projective linear subspace that $M \in \mathbb{C}^{m \times (n+1)} $ defines in $ \mathbb{P}^n$.
We write $[x_0 \colon x_1 \colon \dots \colon x_{n}]$ for homogeneous coordinates in $\mathbb{P}^n$.

\subsection{Intersecting a variety with a varying linear subspace}\label{ss:kappaPeter}

We briefly review some relevant ideas in \cite{Peter2017}, on which our condition number in Section~\ref{ss:kappaFP} is based.
Let $\mathcal{Z}$ be a degree-$p$ variety in $\mathbb{P}^n$ of codimension $n-d$. The \emph{Hurwitz variety} of $\mathcal{Z}$ is a subvariety of the
Grassmannian $\mathbb{G}(n-d,n)$  of $(n-d)$-dimensional projective linear subspaces in $\mathbb{P}^n$
 defined by
\[
\cHu_\mathcal{Z} \coloneqq \{\mathbb{M} \in \mathbb{G}(n-d,n) :  \mathcal{Z}\cap \mathbb{M} \; \text{does not consist of $p$ reduced points}\}.
\]
The Hurwitz variety  is an irreducible hypersurface defined by a polynomial called the \emph{Hurwitz form}  in the coordinate ring of the Grassmannian \cite{Sturmfels2017}.
This variety can be regarded as the set of ill-posed instances of
the problem of intersecting a variety by a varying linear space \cite{Peter2017}. By \cite[Definition~1.1 and Theorem~1.4]{Peter2017}, we have the following definition.
\begin{definition}\label{def:intersectionConditionNumber}
Let $\mathcal{Z}$ be a  $d$-dimensional irreducible projective variety in $\mathbb{P}^n$ and  $\mathbb{M}\in\mathbb{G}(n-d,n)$. Let $\zv\in\mathcal{Z} \cap \mathbb{M}$
and $\alpha$ be the minimum angle between the tangent spaces $\mathsf{T}_{\zv}( \mathcal{Z})$ and $\mathsf{T}_{\zv} (\mathbb{M})$.
The \emph{intersection condition number} of $\mathbb{M}$ at $\zv$ with respect to $\mathcal{Z}$ is
\[
\kappa_\mathcal{Z}(\mathbb{M},\zv) \coloneqq
\frac{1}{\sin\alpha}
\]
if $\zv$ is a smooth point of $\mathcal{Z}$ and $\mathbb{M}$ intersects $\mathcal{Z}$ transversally at $\zv$, and $\kappa_\mathcal{Z}(\mathbb{M},\zv) \coloneqq \infty$ otherwise.
\end{definition}
The elements of $\cHu_\mathcal{Z}$ are precisely the projective linear subspaces where the intersection condition number is infinite  \cite[Theorem~1.6]{Peter2017}.

\subsection{Condition number of the fiber product MEP}\label{ss:kappaFP}

Consider the multiparameter eigenvalue variety $\EV(H_1,\dots,H_k) \subseteq \mathbb{C}^{k^2}$ in \eqref{eq:EVV}. Throughout this section, $\mathcal{Z} \subseteq \mathbb{P}^{k^2}$ will always denote the  projective variety
\begin{equation}\label{eq:projectiveEV}
\mathcal{Z} \coloneqq
\overline{
\{
[1 \colon \Lam_1 \colon  \dots \colon \Lam_k] \in \mathbb{P}^{k^2} \colon
(\Lam_1,\dots,\Lam_k)\in \EV(H_1,\dots,H_k)
\}},
\end{equation}
where we have denoted the homogeneous coordinates of $\mathbb{P}^{k^2}$ by
\[
[\lambda_0 \colon \underbrace{\lambda_{11} \colon \dots \colon \lambda_{1k}}_{\Lam_1} \colon \dots \colon  \underbrace{\lambda_{k1} \colon \dots \colon \lambda_{kk}}_{\Lam_k}],
\]
with $\lambda_0$ as the homogenizing coordinate.
The defining equations of $\mathcal{Z}$ are homogeneous polynomials in $\lambda_0,\Lam_1,\dots,\Lam_k$ such that  setting  $\lambda_0 =1$ gives the defining equations for $\EV(H_1,\dots,H_k)$ in \eqref{eq:EVV}

We will apply the notion of an intersection condition number in Section~\ref{ss:kappaPeter} to define a condition number for fiber product MEP.
\begin{definition}\label{def:condFP}
Given a fiber product multiparameter eigenvalue problem \eqref{eq:equivMEP}, let  $\mathcal{Z}$ be the projective variety in \eqref{eq:projectiveEV}, $\mathbb{M}_\FP $ be the projective linear space
\begin{align}
\mathbb{M}_\FP &\coloneqq \overline{\{ [1 \colon \Lam_1 \colon \dots \colon \Lam_k]\in\mathbb{P}^{k^2} :  \Lam_1=\cdots=\Lam_k\}} \notag \\
&= \{ [\lambda_0 \colon \Lam_1 \colon \dots \colon \Lam_k]\in\mathbb{P}^{k^2} : \Lam_1=\cdots=\Lam_k\}, \label{eq:MFP}
\end{align}
and $\zv $ be the point
\[
\zv \coloneqq [1\colon \Lam_1\colon \dots\colon \Lam_k] \in\mathcal{Z}\cap\mathbb{M}_\FP.
\]
The \emph{condition number of the fiber product MEP} \eqref{eq:equivMEP} is given by
\[
\kappa_\FP(\Lam_1,\dots,\Lam_k,H_1,\dots,H_k) \coloneqq \kappa_\mathcal{Z}(\mathbb{M}_\FP,\zv).
\]
\end{definition}
Note that $\mathcal{Z}\cap\mathbb{M}_\FP$ consists precisely of points $[1\colon \Lam_1\colon \dots\colon \Lam_k] \in\mathbb{P}^{k^2}$ where $(\Lam_1,\dots,\Lam_k )\in \mathbb{C}^{k^2}$ are the eigenvalues of the MEP.
We next show how to compute $ \kappa_\mathcal{Z}(\mathbb{M},\zv) $ for any projective linear subspace $\mathbb{M}$ --- as we will see later, for fiber product homotopy method, we would also be interested in $ \kappa_\mathcal{Z}(\mathbb{M},\zv) $ for projective linear subspaces $\mathbb{M}$ other than $\mathbb{M}_\FP$.
\begin{theorem}\label{theorem:JacZEV}
Let  $A_{ij} \in \mathbb{C}^{n_i \times n_i}$, $i=1,\dots,k$,
$j=0,1,\dots,k$, be the matrices of an MEP as in Definition~\ref{definition:MEP}.
Let  $\mathcal{Z}$ be the projective variety in \eqref{eq:projectiveEV}, $\mathbb{M}=\pi(\ker(M)\setminus\{0\})$ for some $M\in\mathbb{C}^{(k^2-k)\times(k^2+1)}$, and
$\zv = [1 \colon \Lam_1 \colon \dots\colon \Lam_k] \in\mathcal{Z}\cap\mathbb{M}$
be a smooth point of $\mathcal{Z}$. If $\xv_i,\yv_i \in\mathbb{C}^{n_i}$ are nonzero vectors\footnote{These are respectively right and left eigenvectors of the polynomial matrix.} such that
\[
H_i(\Lam_i)\xv_i = 0, \qquad H_i(\Lam_i)^\tp\yv_i = 0, \qquad i=1,\dots,k,
\]
then the intersection condition number
\[
 \kappa_\mathcal{Z}(\mathbb{M},\zv) = \frac{1}{\sin \alpha},
\]
where $\alpha$ is the minimum angle between $\mathbb{M}$
and $\mathbb{J} =\pi\bigl(\ker(J)  \setminus \{0\}\bigr)$,
\begin{equation}\label{eq:jacZ}
J \coloneqq
\begin{bmatrix}
\yv_1^\tp A_{10}\xv_1   & \yv_1^\tp B_1(\xv_1)\\
\yv_2^\tp A_{20}\xv_1   && \yv_2^\tp B_2(\xv_2)\\
\vdots&&&\ddots\\
\yv_k^\tp A_{k0}\xv_k&&&&\yv_k^\tp B_k(\xv_k)
\end{bmatrix}\in \mathbb{C}^{k\times(1+k^2)}
\end{equation}
with  $B_i(\xv_i) \coloneqq - [A_{i1}\xv_i,\dots,A_{ik}\xv_i] \in \mathbb{C}^{n_i \times k}$, $i =1,\dots,k$.
\end{theorem}
\begin{proof}

By Definition~\ref{def:intersectionConditionNumber}, it suffices to show that $\mathsf{T}_{\zv}( \mathcal{Z})=\mathsf{T}_{\zv}(\mathbb{J})$.  First, recall from Definition~\ref{def:epvar} that the multiparameter eigenvalue variety $\EV(H_1,\dots,H_k)$
 is a projection of the multiparameter eigenpair variety $\EP(H_1,\dots,H_k)$.
So the projective closure of $\EP(H_1,\dots,H_k)$ in
$\mathbb{P}^{k^2}\times\mathbb{P}^{n_1-1}\times\cdots\times\mathbb{P}^{n_k-1}$, which we will denote
by $\overline{\EP}(H_1,\dots,H_k)$, projects onto the variety $\mathcal{Z}$.

For a linear polynomial matrix
\[
H(\Lam) = A_{0}-\lambda_{1}A_{1}-\dots-\lambda_{k}A_{k},
\]
we  introduce a homogenizing variable $\lambda_0$ and define a homogenized $H$ as
\[
\widehat{H}(\lambda_0, \Lam)\coloneqq \lambda_0 A_{0}-\lambda_{1}A_{1}-\dots-\lambda_{k}A_{k}.
\]
We homogenize $H_1,\dots,H_k$ in this manner, write $\wv \coloneqq (\lambda_0, \Lam_1,\dots,\Lam_k,\xv_1,\dots,\xv_k)$ for brevity, and then let
\[
\boldsymbol{H} (\wv) \coloneqq
\begin{bmatrix}
\widehat{H}_1(\lambda_0,\Lam_1)\xv_1\\
\vdots\\
\widehat{H}_k(\lambda_0,\Lam_k)\xv_k
\end{bmatrix}.
\]
Write $\boldsymbol{H} = [ h_1,\dots,h_{n_1 + \dots + n_k}]^\tp$
where $h_{i} \in \mathbb{C}[\wv]$, $i =1,\dots,n_1 +\dots + n_k$, are \emph{multihomogeneous polynomials} --- homogeneous in the variables $\lambda_0, \Lam_1,\dots,\Lam_k$ and homogeneous in each of the variables $\xv_1,\dots,\xv_k$.
Note that the $h_i$'s vanish on the variety  $\overline{\EP}(H_1,\dots,H_k)$; in fact, the set of  polynomials that vanish on $\overline{\EP}(H_1,\dots,H_k)$ is given by
\[
\{h\in\mathbb{C}[\wv] :
\lambda_0^m h=g_1h_1+\cdots +g_{n_1 + \dots +n_k} h_{n_1 + \dots +n_k}, \; g_i\in\mathbb{C}[\wv],
\; m\in\mathbb{N}
\}.
\]

The $(n_1+\dots+n_k)\times(1+k^2+n_1+\dots+n_k)$ Jacobian matrix,
\[
\begin{medsize}
\Jac \boldsymbol{H} (\wv) =
\biggl[ \frac{\partial }{\partial \lambda_0}\boldsymbol{H} , \Jac_{\Lam_1}\boldsymbol{H} , \dots,  \Jac_{\Lam_k}\boldsymbol{H} ,  \Jac_{\xv_1}\boldsymbol{H}, \dots,  \Jac_{\xv_k}\boldsymbol{H}  \biggr],
\end{medsize}
\]
when  evaluated at a  point
\[
\wv_0 \coloneqq
([1:\Lam_1:\dots:\Lam_k],\xv_1,\dots,\xv_k)=(\zv,\xv_1,\dots,\xv_k)\in\overline{\EP}(H_1,\dots,H_k),
\]
has the form
\[
\begin{medsize}
\Jac \boldsymbol{H} (\wv_0) =
\begin{bmatrix}
 A_{10}\xv_1   &  B_1(\xv_1)&&&&\widehat{H}_1(1,\Lam_1)\\
 A_{20}\xv_1   &&  B_2(\xv_2)&&&&\widehat{H}_2(1,\Lam_2)\\
\vdots&&&\ddots&&&&\ddots\\
 A_{k0}\xv_k&&&& B_k(\xv_k)&&&&\widehat{H}_k(1,\Lam_k)
\end{bmatrix}.
\end{medsize}
\]
For any $L\in\mathbb{C}^{m\times(k^2+1)}$,
if each row of the matrix
$[L,\mathbbold{O}_{m\times(n_1+\cdots+n_k)}]\in\mathbb{C}^{m\times(k^2+1+n_1+\cdots+n_k)}$ is in the row space of $\Jac \boldsymbol{H} (\wv_0)$,
then $\mathbb{L} =\pi\bigl(\ker(L)  \setminus \{0\}\bigr)$ is such that
$\mathsf{T}_{\zv}( \mathcal{Z})\subseteq \mathsf{T}_{\zv}(\mathbb{L})$.
This is because $\mathcal{Z}$ is the projection of $\overline{\EP}(H_1,\dots,H_k)$ to $\mathbb{P}^{k^2+1}$ in the first $k^2+1$ coordinates.
With this in mind, we multiply $\Jac \boldsymbol{H} (\wv_0)$ on the left by
\[
\begin{medsize}
\begin{bmatrix}
 \yv_1^\tp\\
& \yv_2^\tp\\
&& \ddots\\
&&& \yv_k^\tp
\end{bmatrix}\in \mathbb{C}^{k\times(n_1+\dots+n_k)}
\end{medsize}
\]
to obtain the $k\times(1+k^2+n_1+\dots+n_k)$ matrix
\[
\begin{medsize}
\begin{bmatrix}
\yv_1^\tp A_{10}\xv_1   & \yv_1^\tp B_1(\xv_1) &&&&\mathbbold{O}_{n_1+\dots+n_k}^\tp\\
\yv_2^\tp A_{20}\xv_1   && \yv_2^\tp B_2(\xv_2)&&&\mathbbold{O}_{n_1+\dots+n_k}^\tp\\
\vdots&&&\ddots&&\vdots\\
\yv_k^\tp A_{k0}\xv_k&&&&\yv_k^\tp B_k(\xv_k)&\mathbbold{O}_{n_1+\dots+n_k}^\tp
\end{bmatrix} = [J, \;\mathbbold{O}_{k \times (n_1+\dots+n_k)}],
\end{medsize}
\]
showing that $ \mathsf{T}_{\zv}( \mathcal{Z})\subset\mathsf{T}_{\zv}(\mathbb{J})$.
When $J$ is full rank, the codimensions
\[
\codim \mathsf{T}_{\zv}(\mathbb{J}) = k = \codim \mathsf{T}_{\zv}( \mathcal{Z})
\]
when $\zv$ is a smooth point. It follows that $\mathsf{T}_{\zv}( \mathcal{Z})=\mathsf{T}_{\zv}(\mathbb{J})$.
\end{proof}

We show how we obtain the condition number $ \kappa_\FP(\Lam_1,\dots,\Lam_k,H_1,\dots,H_k)$.  Consider the following one-parameter family of projective linear spaces induced by a one-parameter family of matrices:
For a fixed $t \in [0,1]$, let
\begin{equation}\label{eq:Mtwiddle}
\begin{medsize}
M_t\coloneqq
(1-t)\begin{bmatrix}
-\mathbbold{1}_{k-1} & \Jac L_{1}\\
-\mathbbold{1}_{k-1} &  & \Jac L_{2}\\
\vdots &  &  & \ddots\\
-\mathbbold{1}_{k-1} &  &  &  & \Jac L_{k}
\end{bmatrix}+t\begin{bmatrix}
\mathbbold{O}_{k-1} & \Jac G_{1}\\
\mathbbold{O}_{k-1} &  \Jac G_{2}\\
\vdots & \vdots\\
\mathbbold{O}_{k-1} &   \Jac G_{k}
\end{bmatrix}\in \mathbb{C}^{k(k-1)\times(k^2+1)},
\end{medsize}
\end{equation}
where $L_1,\dots, L_k$ are as in \eqref{eq:solveProduct} and  $G_1,\dots,G_k$ are as in \eqref{eq:defineG}. Let
\[
\mathbb{M}_t \coloneqq \pi\bigl(\ker(M_t)  \setminus \{0\}\bigr) \subseteq \mathbb{P}^{k^2}
\]
be the corresponding projective linear subspace. Note that $\mathbb{M}_t$ dehomogenizes (by setting $\lambda_0 = 1$) to the affine linear space in $\mathbb{C}^{k^2}$ defined by
\begin{alignat*}{3}
0&=(1-t)L_{1}(\Lam_{1})&&+tG_{1}(\Lam_{1},\dots,\Lam_{k}),\\
\MTFlushSpaceAbove
&\vdotswithin{=}&& \vdotswithin{+}
\MTFlushSpaceBelow
0&=(1-t)L_{k}(\Lam_{k})&&+tG_{k}(\Lam_{1},\dots,\Lam_{k}),
\end{alignat*}
where the right-hand sides are the last $k(k-1)$ linear polynomials in \eqref{eq:homotopyPQ}.
Therefore, for any $t\in[0,1)$, the projective closure of the set of solutions to $\boldsymbol{H}_\FP(\Lam_1,\dots,\Lam_k,\xv_1,\dots,\xv_k,t) =0$ defined in \eqref{eq:homotopyPQ}
projects onto $\mathcal{Z}\cap\mathbb{M}_t$.

When $t=1$, by \eqref{eq:G=} and \eqref{eq:MFP}, we have
\[
\mathbb{M}_1=\{
[\lambda_0 \colon \Lam_1 \colon \dots \colon \Lam_k] \in\mathbb{P}^{k^2} :  \Lam_1=\cdots=\Lam_k
\} = \mathbb{M}_\FP,
\]
and for $\zv = [1 \colon \Lam_1 \colon \dots\colon \Lam_k] \in\mathcal{Z}\cap \mathbb{M}_1$, we obtain, by Definition~\ref{def:condFP},
\[
\kappa_\mathcal{Z}(\mathbb{M}_1,\zv)= \kappa_\FP(\Lam_1,\dots,\Lam_k,H_1,\dots,H_k).
\]
The intersection condition number $\kappa_\mathcal{Z}(\mathbb{M}_t,\zv)$ for  $t \in [0,1)$ is also useful as it measures conditioning of the subproblems encountered during path tracking in the fiber product homotopy method. The next theorem shows that $\kappa_\mathcal{Z}(\mathbb{M}_t,\zv)$ is almost always finite and Example~\ref{eg:cond1} indicates that it is typically small.
\begin{theorem}
For any $t\in[0,1)$, $\kappa_\mathcal{Z}(\mathbb{M}_t,\zv)$ is finite with probability one.
\end{theorem}
\begin{proof}
Since the Hurwitz variety $\cHu_\mathcal{Z}$ comprises projective linear subspaces $\mathbb{M}$ with $\kappa_\mathcal{Z}(\mathbb{M},\zv) = \infty$, it suffices to show that for any $t\in[0,1)$, $\mathbb{M}_t\notin\mathcal{H}_\mathcal{Z}$ with probability one.
By Theorem~\ref{thm:chosenCorrectly}, the fiber product homotopy
$\boldsymbol{H}_\FP \colon \mathbb{C}^{k^2}\times(\mathbb{P}^{n_1-1}\times\dots\times\mathbb{P}^{n_k-1})\times\mathbb{C}\to\mathbb{C}^{n_1+\dots +n_k}\times\mathbb{C}^{k(k-1)}$
has a start system chosen correctly with probability one.
Thus  $\boldsymbol{H}_\FP(\Lam_1,\dots,\Lam_k,\xv_1,\dots,\xv_k,t) =0$, $t\in[0,1)$, has smooth solution paths with probability one.

If  $\mathbb{M}_t\in\mathcal{H}_\mathcal{Z}$,
then these solution paths would not be smooth as the projective closure of the set of solutions to $\boldsymbol{H}_\FP(\Lam_1,\dots,\Lam_k,\xv_1,\dots,\xv_k,t) =0$ projects onto $\mathcal{Z}\cap\mathbb{M}_t$.
Thus, with probability one, $\kappa_\mathcal{Z}(\mathbb{M}_t,\zv)$ is finite for $t\in[0,1)$.
\end{proof}

For comparison,  $\kappa(\Lam,H_1,\dots,H_k) $, the condition number of a standard MEP \eqref{eq:mep1} as defined in \cite{HP2003}  captures how small perturbations in the inputs $A_{ij}$, $  i=1,\dots, k, j=0,\dots, k$, affect the  eigenvalue $\Lam = (\lambda_1,\dots,\lambda_k)$, i.e.,
\[
\begin{multlined}
\limsup\limits_{\varepsilon \rightarrow0}
\Bigl\{\|\Delta \Lam\|/\varepsilon
: ( A_{i0}+\Delta A_{i0}-\sum\nolimits_{j=1}^k(\lambda_j+\Delta \lambda_j)(A_{ij}+\Delta A_{ij})) (\xv_i+\Delta \xv_i)=0,\\
 \|\Delta A_{ij}\|\leq \varepsilon\|A_{ij}\|, i=1,\dots, k, j=0,\dots, k \Bigr\}
\end{multlined}
\]
where $(\xv_1,\dots,\xv_k)$ is the eigenvector, i.e., $H_i(\Lam) \xv_i =0$, $i =1,\dots,k$. We will also let $(\yv_1,\dots,\yv_k)$ be the  \emph{left} eigenvector, i.e., $H_i(\Lam)^\tp\yv_i =0$,  $i =1,\dots,k$.

Let  $\theta_i\coloneqq \|A_{i0}\|+\sum_{j=1}^k|\lambda_j|\|A_{ij}\|$, $i=1,\dots ,k$. The \emph{$\theta$-weighted norm},
 \[
\|M \|_\theta \coloneqq \max \{\|M \zv \|_2:\zv\in \mathbb{C}^k,\; |z_i|=\theta_i,\; i=1,\dots,k\},
\]
for any $M \in \mathbb{C}^{k \times k}$.
By \cite[Theorem~6]{HP2003}, $\kappa(\Lam,H_1,\dots,H_k)=\|M^{-1}\|_{\theta}$ where
\[
M \coloneqq \begin{bmatrix}
\yv_1^*A_{11}\xv_1 &\yv_1^*A_{12}\xv_1&\dots&\yv_1^*A_{1k}\xv_1\\
\yv_2^*A_{21}\xv_2 &\yv_2^*A_{22}\xv_2&\dots&\yv_2^*A_{2k}\xv_2\\
\vdots&\vdots&&\vdots\\
\yv_k^*A_{k1}\xv_k &\yv_k^*A_{k2}\xv_k&\dots&\yv_k^*A_{kk}\xv_k
\end{bmatrix}.
\]

We next show that the same Matheiu problem formulated as a fiber product MEP \eqref{eq:equivMEP} and a standard MEP \eqref{eq:mep1} can have vastly different condition numbers.
\begin{example}[Conditioning of Mathieu problem]\label{eg:cond1}
We generate an instance of the Mathieu problem in Seciton~\ref{sec:mathieu} with $n_1=18$ and $n_2=38$.
In Figure~\ref{kappa_Mathieu}, we plot the fiber product multiparameter eigenvalue problem condition number $\kappa_\FP(\Lam_1,\Lam_2,H_1,H_2) = \kappa_\mathcal{Z}(\mathbb{M}_1 ,\zv)$ and the standard multiparameter eigenvalue problem condition number $\kappa(\Lam,H_1,H_2)$,  against   $\Lam_1$ and $\Lam$  respectively in increasing norms.
\begin{figure}[htb]
\centering
\caption{The condition numbers (vertical axis, $\log_{10}$-scale) $\kappa_\FP(\Lam_1,\Lam_2,H_1,H_2)$ for fiber product homotopy (blue) and $\kappa(\Lam,H_1,H_2)$ for Delta method (red) plotted respectively against eigenvalues $\Lam_1$ and $\Lam$ (horizontal axis), ordered in increasing norm.}
\includegraphics[trim={0ex, 0ex, 0ex, 0ex}, clip,width=0.75\textwidth]{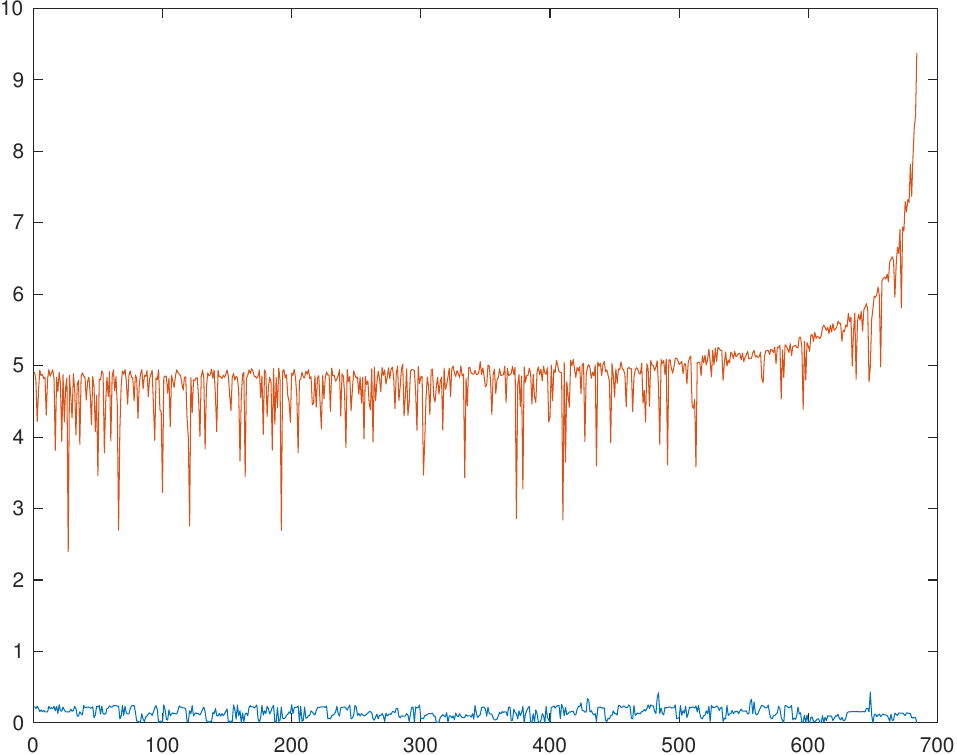}
\label{kappa_Mathieu}
\end{figure}

The difference is striking --- we emphasize that the vertical axis of the figure is in \blue{$\log_{10}$}-scale. The values of $\kappa_\FP(\Lam_1,\Lam_2,H_1,H_2) $ are all \blue{close to $1$ and certainly} less than \blue{$10$}. So small perturbations of the projective linear space $\mathbb{M}_1$ yield small changes in the points of intersection of $\mathcal{Z}$ with $\mathbb{M}_1$, which correspond to the eigenvalues we seek. On the other hand, the values of $\kappa(\Lam,H_1,H_2)$ are vastly larger, \blue{with a majority exceeding} $10^5$ and those near the right end of the plot corresponding to the largest eigenvalues are as large as \blue{$10^{9}$}.

Apart from $ \kappa_\mathcal{Z}(\mathbb{M}_1 ,\zv)$, we have also computed the intersection condition number $\kappa_\mathcal{Z}(\mathbb{M}_t,\zv)$ as $t$ varies from $0$ to $1$ in every path we tracked. We ran our implementation five times and found that $\kappa_\mathcal{Z}(\mathbb{M}_t ,\zv)$ does not exceed $21.85$ for every $t$ we encountered in every path and in every run.
\end{example}

\section{Conclusions}\label{sec:conclude}

The fiber product homotopy method solves an MEP by solving a mathematically equivalent problem, the fiber product MEP, via a homotopy algorithm  designed to exploit its structure. Our numerical experiments show that the fiber product homotopy method: (i) is much faster than diagonal coefficient homotopy method on all instances and is faster than the Delta method on large instances; (ii) is extremely accurate, producing relative backward errors on the order of $10^{-16}$,  especially in comparison with the $10^{-11}$ errors in the Delta method; (iii) is unique in that it maintains the same high degree of accuracy across a robust range of parameters --- irrespective of the dimensions of the matrices,  magnitudes of the eigenvalues, or  singularity of the MEP.

We proffer two insights to explain its strength:
(a) it deforms \blue{only linear} equations \blue{whereas} the diagonal coefficient homotopy method \blue{deforms nonlinear equations};
(b) the problem that it solves, the fiber product MEP, is much better conditioned than the equivalent standard MEP with the same solutions.

\section*{Acknowledgment}
We thank the anonymous reviewer for very pertinent comments and helpful suggestions.
The work in this article is generously supported by DARPA D15AP00109 and NSF IIS 1546413. LHL is supported by a DARPA Director's Fellowship and the Eckhardt Faculty Fund.

\end{document}